\documentclass{amsart}
\usepackage[utf8]{inputenc}
\usepackage{mathtools}
\usepackage{accents}

\usepackage{amssymb}
\usepackage{amsmath}
\usepackage{amsthm}
\usepackage{graphicx}
\usepackage{esint} 
\usepackage{url}
\usepackage[normalem]{ulem} 

\newcommand{\RR}{\mathbb R}

\newcommand{\NN}{\mathbb N}
\newcommand{\ZZ}{\mathbb Z}
\renewcommand{\SS}{\mathbb S}
\newcommand{\BB}{\mathbb B}
\renewcommand{\H}{\mathcal H}
\newcommand{\eps}{\varepsilon}
\newcommand{\loc}{\mathrm{loc}}
\DeclareMathOperator{\Per}{Per}

\renewcommand{\vec}[1]{\mathbf{#1}}
\newcommand{\abs}[1]{\left\vert #1 \right\vert}
\newcommand{\Abs}[1]{\left\Vert #1 \right\Vert}
\newcommand{\enclose}[1]{\left(#1\right)}
\newcommand{\Enclose}[1]{\left[#1\right]}
\newcommand{\ENCLOSE}[1]{\left\{#1\right\}}

\newcommand{\defeq}{:=}
\renewcommand{\subset}{\subseteq}

\newtheorem{theorem}{Theorem}[section]
\newtheorem{proposition}[theorem]{Proposition}

\newtheorem{lemma}[theorem]{Lemma}
\newtheorem{corollary}[theorem]{Corollary}

\theoremstyle{definition}
\newtheorem{definition}[theorem]{Definition}
\theoremstyle{remark}
\newtheorem{remark}[theorem]{Remark}

\title{    
Existence of a non-standard isoperimetric triple partition
}

\author[]{Matteo Novaga} 
\address{Universit\`a di Pisa, Dipartimento di Matematica, Largo Bruno Pontecorvo 5,
56127 Pisa, Italy}
\email{\url{matteo.novaga@unipi.it}}
\author[]{Emanuele Paolini} 
\address{Universit\`a di Pisa, Dipartimento di Matematica, Largo Bruno Pontecorvo 5,
56127 Pisa, Italy}
\email{\url{emanuele.paolini@unipi.it}}
\author[]{Vincenzo M. Tortorelli} 
\address{Universit\`a di Pisa, Dipartimento di Matematica, Largo Bruno Pontecorvo 5,
56127 Pisa, Italy}
\email{\url{vincenzo.tortorelli@unipi.it}}

\thanks{
 This research was partially supported by MUR Excellence Department Project 
 awarded to the Department of Mathematics of the University of Pisa.
 The work of M.N. was partially supported by Next Generation EU, PRIN 2022E9CF89;
 the work of E.P. was partially supported by Next Generation EU, PRIN 2022PJ9EFL.
 The first and second authors are members of the INDAM-GNAMPA}

\date\today 

\begin{document}

\begin{abstract}
We show existence of a isoperimetric $3$-partition of $\RR^8$, with one set of finite volume and two of infinite volume, which is asymptotic to a singular minimal cone.
\end{abstract}

\maketitle

\tableofcontents

\section{Introduction}

An \emph{$M$-cluster} in $\RR^d$ is a $M$-uple $\vec E = (E_1,\dots,E_M)$ of pairwise disjoint subsets of $\RR^d$.
We say that $\vec E$ is an \emph{isoperimetric cluster}
if the surface area of its boundary $\H^{d-1}(\partial \vec E) = \H^{d-1}(\partial E_1 \cup \dots \cup \partial E_M)$ is minimal among all $M$-clusters $\vec F$ which have regions with the same volumes: $\abs{F_k}=\abs{E_k}$, $k=1,\dots,M$. 
In the case $M=1$ clusters are single sets and isoperimetric clusters are isoperimetric sets, i.e., balls.

Similarly, a \emph{$N$-partition} is a $N$-uple 
$\vec E=(E_1,\dots, E_N)$ of pairwise disjoint subsets $\RR^d$ whose union is the whole space $\RR^d$.
If all but one of the regions are bounded, the interface $\partial \vec E=\partial E_1\cup\dots\cup\partial E_N$ can have finite 
surface area, whilst if two regions are unbounded the total area must be infinite.
Recently, in \cite{AlaBroVri25}, the natural concept of \emph{isoperimetric partition} is introduced: 
a partition $\vec E$ is \emph{isoperimetric} if it is \emph{locally isoperimetric} which means 
that $\H^{d-1}(\partial \vec E\cap B)\le \H^{d-1}(\partial \vec F\cap B)$ whenever $B$ is bounded 
and $\vec F$ is an $N$-partition such that $\abs{E_k}=\abs{F_k}$ and $E_k\triangle F_k\Subset B$,
for $k=1,\dots, N$.
If a $N$-partition has a single unbounded region, the bounded regions compose an $M$-cluster
with $M=N-1$ and the $N$-partition is isoperimetric if and only if the corresponding $M$-cluster is isoperimetric.
On the other hand, if all the regions of an $2$-partition are unbounded, the isoperimetric constraint on the volumes 
becomes empty (see \cite{BroNov24}) so that isoperimetric $2$-partitions correspond to \emph{locally minimal sets}.

A lens partition $\vec L=(L_1,L_2,L_3)$ in $\RR^d$ is a partition where the bounded region $L_1$ is the intersection of two balls 
with the same radius $R$ and centers $\vec p_1,\vec p_2$ which have distance $R$, and the two unbounded regions $L_2$ and $L_3$ have common
boundary on the hyperplane which has the same distance from the points $\vec p_1, \vec p_2$.
In \cite{AlaBroVri25} it is proven that a $3$-partition of $\RR^2$ with a single bounded region is isoperimetric if and only if it is 
a lens partition. 
The analogous result in $\RR^d$, for $d\le 7$, has been obtained in \cite{NovPaoTor25} (existence) and \cite{BroNov24} (uniqueness), 
where the authors suggest that there might be a counterexample in $\RR^8$, in analogy with what happens for locally minimal sets. In this paper we provide such example.

If $M\le d+1$, $N=M+1$ there is a \emph{standard} way to partition the sphere $\SS^d$ into $N$ equal regions. 
Take $N$ mutually equidistant points $(\vec p_0, \vec p_1,\dots, \vec p_N) \in \SS^d$, 
and consider the corresponding Voronoi partition of $\SS^d$ i.e., the partition $\vec V$ 
with regions $V_k = \ENCLOSE{\vec x\in \SS^d\colon \abs{\vec x-\vec p_k} \le \abs{\vec x - \vec p_k}, j=1,\dots, N}$.
The stereographic projection $\pi\colon\SS^d \to \RR^d\cup\ENCLOSE{\infty}$ 
sends the Voronoi partition $V$ into a partition of $\RR^d$. 
If we suppose that the {north-pole} $\nu=\pi^{-1}(\infty)\in \SS^d$ is an interior point of the region $V_0$, the 
regions $E_k=\pi(V_k)$ for $k=1,\dots,M=N-1$ are bounded and together compose an $M$-cluster $\vec E$ which is called a \emph{standard cluster}.
On the other hand if the north pole happens to lie on the boundary of two or more regions, the stereographic projection gives 
an $N$-partition of $\RR^d$ which is not corresponding to an $M$-cluster. 
We call these partitions \emph{standard partitions}.

A long standing conjecture by J. Sullivan states that, for $M\le d+1$, isoperimetric $M$-clusters coincide with standard clusters.
The conjecture was confirmed in the case $M=2$, $d=2$ by \cite{Foi93}, in the case $M=2$, $d=3$ 
by \cite{HutMorRitRos02}, in the case $M=2$, $d\ge 4$ by \cite{Rei08}, in the case $M=3$, $d=2$ 
by \cite{Wic04}. 
More recently, in \cite{MilNee23}, the conjecture has been confirmed 
in the cases $M=3$, $d\ge 3$ and $M=4$, $d\ge 4$.
Moreover in \cite{MilXu25} the authors prove that standard $N$-partitions (for all $N\le d+2$)
and in particular standard $M$-clusters (for $M\le d+1$)
are \emph{stable} if $d\ge 3$.
In \cite{PaoTor20,PaoTam16} an example of an isoperimetric cluster,   with $d=2,\, M=4$, necessarily non-standard, is given.
To our knowledge there are no other proven examples of isoperimetric clusters.

In \cite{NovPaoTor25} we prove that any partition which is obtained as the limit of standard isoperimetric clusters is standard and isoperimetric. 
As a consequence, recalling the result in \cite{MilNee22}, we obtain that standard partitions are isoperimetric for all $N\le \min(d+1,5)$.
In the case $N=3$, $d\le 7$ in \cite{BroNov24} the authors prove a uniqueness result which ensures that all isoperimetric partitions are standard.
If $d\ge 8$ there are examples (the first one being the Simons' cone in $d = 8$) of 
locally minimal sets in $\RR^d$ which are singular cones: 
in our terminology these are $2$-partition which are non-standard.
We also point out that in \cite{Bra91} there are examples of $N$-partitions of $\RR^d$ which are isoperimetric but non-standard for $N=2d\ge 8$ and $N=d\ge 6$,
whose boundary is the cone over the $(d-1)$-skeleton of the hypercube.

In the present paper we prove that, also in the case $N=3$, there exists a non-standard isoperimetric partition in dimension $d=8$ (see Theorem~\ref{teomain}),
as was conjectured in \cite{NovPaoTor25} and in \cite{BroNov24} .
Even if the construction is based on the Simons' cone, 
we are not able to show that the non-standard isoperimetric partition that we find is
asymptotic to the Simons's cone. More generally, one could ask whether, given a singular minimal cone in $\RR^d$, there exists 
an isoperimetric $3$-partition asymptotic to that cone.




\begin{figure}
\centering\includegraphics[height=5cm]{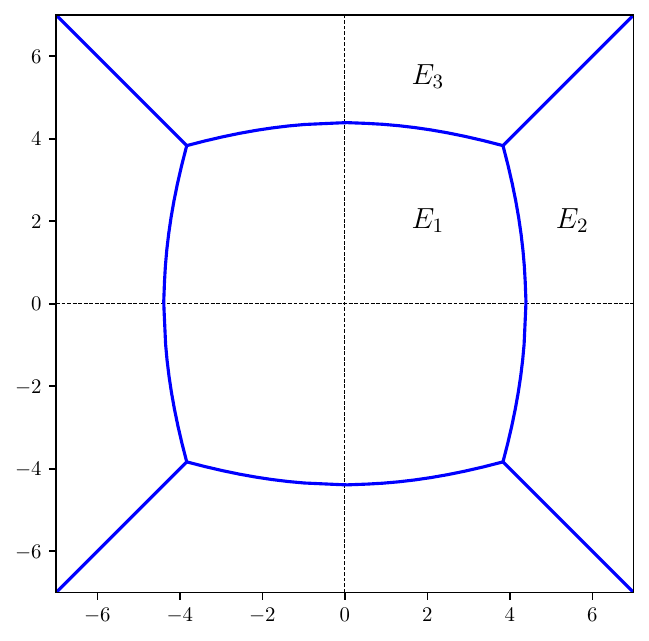}
\label{fig:barrel}
\caption{
Numerically computed picture of the conjectured isoperimetric $3$-partition of $\RR^8$. 
The actual partition in $\RR^8$ would be 
the triple $(r(E_1),r(E_2),r(E_3))$ where $E_1,E_2,E_3\subset \RR^2$ are depicted in the figure,
and $r(E) = \ENCLOSE{(\vec x,\vec y)\in \RR^4\times\RR^4\colon (\abs{\vec x},\abs{\vec y}) \in E}$.
In this picture, where the mean curvature of the bounded region is $1$,
we have $\Per(r(E_1)) \approx 27.91\cdot 10^5$, 
$\abs{r(E_1)} \approx 16.04\cdot 10^5$,
while the perimeter of the missing cone inside $r(E_1)$ is $\approx 9.58\cdot 10^5$ so
that the defect is $\approx 6.82$.
The defect of the \emph{barrel} (i.e.\ the partition with a square in place of $E_1$)
is $\approx 7.10$ and the defect of the lens partition is $\approx 7.29$ (see Lemma~\ref{lm:defect-partition}).
}
\end{figure}

The idea of the proof is the following. 
First, we consider the $3$-partition obtained by adding the cylinder $\BB^4\times \BB^4$ to the Simons' cone, and note that
the perimeter gap required to pass from the Simons' cone to the barrel partition 
(we call it \emph{defect}) is less than 
the perimeter gap required to pass from a hyperplane to a lens partition (see Lemma~\ref{lm:defect-partition}). 
Then we consider the isoperimetric 3-partition which has the boundary data of the Simons' cone on a large sphere, and has a small region  
of prescribed volume. The  defect of such partition is clearly smaller of the defect of the barrel partition.
We then let the radius of the sphere go to infinity and apply concentration compactness to find a possibly infinite set of limit partitions.
Then we notice that, if all these limit partitions were lenses, the total defect, eventually dominated by the defect of the 
isoperimetric partition, would be larger than the defect of a single lens.
This means that at least one of the limit partition should be different from a lens. 
Now we face the major problem in this approach: the fact that limit partitions could \emph{see} the boundary of the large sphere which is going to infinity.
This means that we are forced to study isoperimetric partitions not only in the whole space $\RR^8$, but also in a half-space which is the limit of the spheres.

Of course the barrel partition is not isoperimetric (the triple junctions has an angle of 90 degrees). 
We can numerically find the deformation of the barrel which gives a partition with the same rotational symmetries, the correct angles 
and constant mean curvature (see Figure~\ref{fig:barrel}). 
Proving that this partition is, in fact, isoperimetric, would require some new ideas, because the tools used to prove the minimality of Simons' cone 
are not applicable to isoperimetric partitions.

In the development of our proof we are led to consider isoperimetric partitions contained in a half-space, and
we realized that, when $d\ge 4$, there could exist a $3$-partition in a half-space 
which is not a lens partition.

The plan of the paper is the following: in Section \ref{secnot} we introduce some notation and give the definition of isoperimetric partition.
In Section \ref{seccomp} we recall the concentration-compactness of partitions under natural bounds on the perimeter of the regions (see Theorem \ref{teocon}),
and we prove the closure of isoperimetric partitions (see Theorem \ref{th:teoclosure}).
In Section \ref{secmono} we show the monotonicity formula for isoperimetric partitions in a cone (see Theorem \ref{th:monotonicity-formula}) 
and the existence of blow-down partitions (see Corollary \ref{corinfty}).
Finally, in Section \ref{secexist} we show existence of a non-standard isoperimetric $3$-partition, asymptotic to a singular cone,
which is the main result of this work (see Theorem \ref{teomain}).

\smallskip
After this paper was completed, we were informed by Lia Bronsard that an analogous result has been obtained independently in \cite{BNNS25}.

\section{Notation and preliminary definitions}\label{secnot}

If $A,B$ are subsets of $\RR^d$, we denote with $A\triangle B = (A \setminus B) \cup (B\setminus A)$ the symmetric difference between $A$ and $B$.

For $B\subset \RR^d$ we denote with $\bar B$ and $\partial B$ the topological closure and boundary of $B$.
We denote with $\BB^d=\ENCLOSE{\vec x\in\RR^d\colon \abs{\vec x}<1}$ the unit ball of $\RR^d$
and let $\SS^{d-1}=\partial \BB^d$ be its boundary. We let $\omega_d = \abs{\BB^d}$ be the measure of the unit ball. 
We denote with $B_\rho(\vec x) \defeq \vec x + \rho\BB^d$ the ball of radius $\rho>0$ 
centered at $x\in \RR^d$, and let $B_\rho \defeq B_\rho(\vec 0)$.


A set $C\subset \RR^d$ is said to be a \emph{cone} if for every $t>0$ and $\vec x \in C$ we have $t\vec x\in C$
(a positively homogeneous cone with vertex in the origin).
A \emph{hyperplane} of $\RR^d$ is a $(d-1)$-dimensional vector subspace of $\RR^d$ (passing through the origin).
We denote with $(\vec x,\vec y)$ the scalar product of two vectors $\vec x,\vec y \in \RR^d$.
An \emph{half-space} of $\RR^d$ is any set of the form $\ENCLOSE{\vec x\in \RR^d\colon (\vec x,\nu)> 0}$ for some $\nu\in \RR^d$.
Hyperplanes and half-spaces are cones.
The translation of a hyperplane (resp. half-space) will be called affine hyperplane (resp. half-space).

$\abs{E}$ will denote the Lebesgue measure of a measurable set $E\subset \RR^d$.
Given $E, E_n$ measurable subsets of $\RR^d$ we will write $E_n\stackrel{L^1}\longrightarrow E$ if 
$\abs{E_n \triangle E}\to 0$ as $n\to +\infty$ and $E_n\stackrel{L^1_\loc}\longrightarrow E$ 
if $\abs{(E_n \triangle E)\cap B_r}\to 0$ for all $r>0$.
We also consider on $\RR^d$ the usual Hausdorff, $k$-dimensional, $k\le d$, measures $\H^h$.
If $E\subset \RR^d$ is a measurable set, and $B\subset \RR^d$ is a Borel set, we define the \emph{perimeter} of $E$
in $B$ as:
\[
   \Per(E,B) \defeq \Abs{D 1_E}(B), 
\]
where $\Abs{D 1_E}$ is the total variation of the distributional derivative of the characteristic function of $E$.
We let $\Per(E) \defeq \Per(E, \RR^d)$ be the \emph{perimeter} of $E$.
We say that a measurable set $E\subset \RR^d$ has \emph{finite perimeter} if $\Per(E)<+\infty$,
we say that $E$ is a \emph{Caccioppoli set} if $\Per(E,B_r)<+\infty$ for all $r>0$.
In this case $\Per(E,\cdot)$ is a measure which coincides with $\H^{d-1}\llcorner \partial^* E$ where $\partial^*E$
is the \emph{reduced boundary} of $E$, i.e.\ the set of points of $\partial E$ where an approximate normal vector can be defined.

$H^{(t)}$ denotes the sets of points of $H$ of density $t$ in $H$; when $H$ is a Caccioppoli set,  $\nu_H$ denotes the measure theoretic exterior normal, $\partial^e H$ the essential boundary of $H$, so that for any   $F\in (C^1_c(\RR^d))^d$, we have
\begin{equation*}
\partial^* H\subseteq H^{(\frac12)}\subseteq \partial^e H,\quad
\H^{d-1}(\partial^e H\setminus 
 \partial^*H )=0, \quad 
\int_H \text{\rm div} F \, d\H^d = \int_{\partial^e H}
\langle F\cdot \nu_H\rangle \, d\H^{d-1}.
\end{equation*}



\begin{definition}
For $N\in \NN$ let $E_1,\dots, E_N$ be measurable subsets of $\RR^d$. 
The $N$-uple $\vec E = (E_1, \dots, E_N)$ is said to be an $N$-\emph{partition} of $\RR^d$ if 
$\abs{E_j\cap E_k}=0$ when $j\neq k$ and 
\[
\abs{\RR^d \setminus \bigcup_{j=1}^N E_j} =0.
\]
If each $E_j$ is a Caccioppoli set (has locally finite perimeter) we say that $\vec E$
is a \emph{Caccioppoli partition}.

We say that the $N$-partition $\vec E$ is a \emph{proper} partition if $\abs{E_j}>0$ for all $j=1,\dots,N$, otherwise we might specify that it is  \emph{improper}.

We define 
\[
  \vec m(\vec E) \defeq (\abs{E_1}, \dots, \abs{E_N}) \in [0,+\infty]^d
\]
and 
\[
  \Per(\vec E, B) \defeq \frac 1 2 \sum_{i=1}^N \Per(E_k, B) \in [0,+\infty].
\]

We say that a sequence of $N$-partitions $\vec E_n$ converges to an $N$-partition $\vec F$ in $L^1_\loc$ if 
for all measurable bounded $B\subset \RR^d$ for all $j=1,\dots,N$ we have
\[
  \abs{(E^n_j\triangle F_j)\cap B} \to 0.
\]

Operators on $N$-uples of sets are intended componentwise, as one would understand with usual vectors. 
For example,
if $\vec E$ and $\vec F$ are $N$-partitions in $\RR^d$, $A\subset \RR^d$ and $\vec x \in \RR^d$ we 
adopt the following notation
\begin{align*}
  \vec E \cup A &\defeq (E_1 \cup A, \dots E_N \cup A),
  \\
  \vec E \cup \vec F &\defeq (E_1 \cup F_1, \dots, E_N \cup F_N),
  \\
  \vec E + \vec x &\defeq (E_1+\vec x, \dots, E_n+\vec x),
\end{align*}
where 
\[
  E_k + \vec x = \ENCLOSE{\vec y +\vec x \colon \vec y \in E_k}.
\]
Notice that $\vec E\cup\vec F$ and $\vec E \cup A$ in general are not partitions.
\end{definition}

\begin{definition}[isoperimetric partitions]
Let $X\subset \RR^d$ be a {closed} set and $\vec E$ be an $N$-partition \emph{of} $\RR^d$, $J\subset \ENCLOSE{1,\dots,N}$.
We say that $\vec E$ is \emph{$J$-isoperimetric in $X$} (or simply \emph{$J$-isoperimetric} if $X=\RR^d$)
if for every compact set $K\subset X$  and every partition $\vec F$ of $\RR^d$ such that 
\begin{align}
\label{eq:varloc}
    \vec F \triangle \vec E &\subset K \\
    \label{eq:iso-measure}
    \abs{F_j\cap K} &= \abs{E_j \cap K} \qquad \text{for all $j\in J$}
\end{align}
one has 
\[
  \Per(\vec E, K) \le \Per(\vec F, K).
\]

We say that the partition $\vec E$ is \emph{isoperimetric in $X$} (or simply \emph{isoperimetric} if $X=\RR^d$)
if $\vec E$ is $J$-isoperimetric in $X$ for $J=\ENCLOSE{j\in\ENCLOSE{1,\dots,N}\colon \abs{E_j}<+\infty}$.
\end{definition}


\begin{remark} If $J = \ENCLOSE{j\colon \abs{E_j}<+\infty}$ and ~ \eqref{eq:varloc} holds then 
\eqref{eq:iso-measure} is equivalent to
\begin{align}
 \label{eq:fullmeas}   \abs{F_j} = \abs{E_j} \qquad \text{for all $j=1,\dots,N$}.
 \end{align}
\end{remark}

\begin{remark}
In the case when $X=\RR^d$ and the partition has a single infinite region, by dropping the infinite region we obtain what is usually called a \emph{cluster} and 
isoperimetric partitions correspond to \emph{isoperimetric clusters} (see \cite[Definition 2.3 and Proposition 2.11]{NovPaoTor25}).

When $X=\RR^d$, the notion of $J$-isoperimetric partition given above closely resembles that of \emph{locally minimizing $(N,M)-$clusters} given in \cite[Definition 2.5]{BroNov24}. 
\end{remark}

\section{Compactness of partitions}\label{seccomp}

We extend to partitions of $\RR^d$ the results proven in \cite[Theorems 3.3, 4.2, Proposition 5.2]{NoPaStTo22} for clusters (i.e.\ partitions with only one region with infinite measure) in a more abstract setting.  
In this section we only prove the statements concerning partitions, 
while some known results are collected in Appendix~\ref{secapp} for the reader's convenience.

\begin{theorem}\label{teocon}
Let $\vec E_n$ be a sequence of $N$-partitions of $\RR^d$,
and let 
\[
 m_j \defeq \limsup_n \abs{E^n_j} \in [0,+\infty], \qquad j=1,\dots,N.
\]
Assume that
\begin{gather*}
\sup_{x\in \RR^d} \sup_{n\in \NN} \Per(\vec E_n,B_1(x)) < +\infty,\\
m_j<+\infty \implies \sup_{n\in \NN} \Per(E^n_j) < +\infty.
\end{gather*}
Then there exists a subsequence of $\vec E_n$, non-relabeled, 
an $N$-partition $\vec F$ of $(\RR^d)^\NN$ 
(i.e.\ $\vec F = (\vec F^i)_{i\in \NN}$, $\vec F^i$ is an $N$-partition of $\RR^d$)
and $\vec x_n^i \in \RR^d$
such that 
\begin{gather}
    \label{eq:cc-infty2}
    \lim_{n\to+\infty} \abs{\vec x_n^i-\vec x_n^j} \to \infty \qquad \text{if $i\ne j$}, 
    \\
    \label{eq:cc-converge2}
    \vec E_n-\vec x_n^i \to \vec F^i \qquad \text{in $L^1_\loc(\RR^d)$ for each $i\in \NN$},
    \\
    \label{eq:cc-measure2}    
    \sum_{i\in \NN} \abs{F^i_j} = \lim_n \abs{E^n_j}, 
    \quad \text{if $m_j<+\infty$}.
\end{gather}
\end{theorem}
\begin{proof} 
Up to subsequences we assume that for  each $j\in\{ 1, \dots, N\}$ there exists $\lim_n \abs{E^n_j} =m_j$.
Let $J\defeq \ENCLOSE{j=1,\dots,N\colon m_j<+\infty}$
and consider the sequence of sets 
\[
  E^n \defeq \bigcup_{j\in J} E^n_j.
\]
Applying Theorem~\ref{th:concentration} to the sequence $E^n$, up to a subsequence, 
we can find $\vec x_n^i$ and $F^i$ 
satisfying~\eqref{eq:cc-infty}, \eqref{eq:cc-converge} and~\eqref{eq:cc-sum} {\emph{i.e.}}
\begin{equation}\label{eq:123451} \sum_i \abs{F^i} = \lim_n \sum_{j\in J} \abs{E^n_j}.
\end{equation}

In view of Lemma~\ref{lm:L1-loc-compactness}, up to a subsequence, we can assume that 
there exist $F^i_j$ such that for all $j=1,\dots, N$ we have $E^n_j-\vec x_n^i \to F^i_j$
in $L^1_\loc$. 
For each $i$ the $N$-uple $\vec F^i \defeq (F^i_1, \dots, F^i_N)$ is a partition of $\RR^d$. 
In particular, up to a negligible set, we have $F^i = \bigcup_{j\in J} F^i_j$ so obtaining from  \eqref{eq:123451} 

\begin{equation}
\label{eq:67891} \sum_i \sum_{j\in J}\abs{F_j^i} = \lim_n \sum_{j\in J} \abs{E^n_j}=\lim_n \abs{E^n}.
\end{equation}

Consider any $R>0$ 
and let $n$ be large enough so that $B_R(\vec x_n^j)$ are pairwise disjoint.
By the semicontinuity of Lebesgue measure applied to $E^n_j-\vec x_n^i\to F^i_j$, 
for all $j\in\{1, \dots ,N\}$, we have
\[
    \sum_{i=0}^M \abs{F_j^i \cap B_R}\\
    = \lim_n \sum_{i=0}^M \abs{E_j^n\cap B_R(\vec x_n^i)}
    = \lim_n \abs{E_j^n \cap \bigcup_{i=0}^M B_R(\vec x_n^i)}
    \le \liminf_n \abs{E_j^n}
\]
hence, for $R\to \infty$ and then $M\to \infty$: 
\begin{equation}
\label{eq:39052}
    \sum_i \abs{F_j^i} 
    \le \lim_M \sup_R \sum_{i=0}^M \abs{F_j^i \cap B_R}
    \le \liminf_n \abs{E^n_j}
    \le m_j,\quad \text{for all } j.
\end{equation}
Finally using, for $j\in J$,~\eqref{eq:67891},  and~\eqref{eq:39052},
\begin{align*}
\lim_n \sum_{j\in J} \abs{E^n_j}
&= \lim_n \abs{E^n}
= \sum_i \abs{F^i}
= \sum_{j\in J}\sum_i \abs{F^i_j}\\
&\le \sum_{j\in J} \liminf_n \abs{E^n_j}
\le \liminf_n \sum_{j\in J} \abs{E^n_j}.
\end{align*}
Since the previous inequality is actually an equality, also in~\eqref{eq:39052} we must 
have an equality for all $j \in J$, and we obtain~ \eqref{eq:cc-measure2}.
\end{proof}

\begin{remark} 
By semicontinuity of the perimeter,   
one also gets for all open sets $B$
\begin{equation}
\label{eq:semi1}
  \sum_{i \in \NN}\Per(F^i_j, B) \le \liminf_n \Per( E^j_n, \bigcup_i (B+\vec x^i_n)), 
  \quad \text{for all } j,
\end{equation}
\end{remark}


\begin{lemma}\label{lm:glueing}
  Let $\vec E$ and $\vec F_i$ be Caccioppoli $N$-partitions in $\RR^d$ for $i=1,\dots,n$.
  Let $R>0$ and $x_i\in \RR^d$, for $i=1,\dots, N$, be such that 
  the balls $B_R(x_i)$ are pairwise disjoint. Let $B$ any other ball or $\RR^d$.
  For all $\rho\le R$ define 
  \[
     \vec G_\rho = 
     (\vec E\setminus \bigcup_{i=1}^n B_\rho(x_i)\cap B)
     \cup 
     \bigcup_{i=1}^n (\vec F^i\cap B_\rho(x_i)\cap B).
  \]
  
  Then the set of $\rho \in (r,R)$ such that 
  \begin{equation}\label{eq:glueing}
    \sum_{i=1}^n \Per(\vec G_\rho(x_i), (\partial B_\rho)\cap B) 
    \le \sum_{i=1}^n \frac{\Abs{\vec m((\vec E\triangle \vec F_i) \cap (B_R(x_i)\setminus B_r(x_i))\cap B)}_1}{2\cdot (R-r)}
  \end{equation}
  has positive measure 
  (where $\Abs{v}_1 = \sum_{k=1}^d \abs{v_k}$ is the $\ell_1$-norm of the vector $v\in \RR^d$).
\end{lemma}

\begin{proof}
For all $0<r<R$, for all $i=1,\dots, n$, for all $k=1,\dots, n$ one has 
(see \cite[Lemma 2.6]{NovPaoTor25}):
  \begin{align*}
  \lefteqn{\int_r^R \Per(G_\rho^k, (\partial B_\rho(x_i))\cap B)\, d\rho}\\
  & = \abs{(E^k \triangle F_i^k)\cap (B_R(x_i)\setminus B_r(x_i))\cap B}.
  \end{align*}
  Summing up on $k=1,\dots,n$, on $i=1,\dots, N$, averaging on $\rho\in (r,R)$, 
  and dividing by two, we obtain
  \begin{equation*}
  \fint_r^R \sum_{i=1}^n \Per(\vec G_\rho,(\partial B_\rho(x_i))\cap B)\, d\rho = \sum_{i=1}^n \frac{\Abs{\vec m((\vec E\triangle \vec F_i) \cap (B_R(x_i)\setminus B_r(x_i))\cap B)}_1}{2\cdot (R-r)}.
  \end{equation*}
  The conclusion follows by noting that the integrand on the left-hand side cannot be 
  larger than its average.
\end{proof}

We recall the following result from \cite[Theorem~2.8]{NovPaoTor25}.

\begin{theorem}
\label{th:volume_fixing_variations_alternative}  
  Let $\Omega$ be an open subset of $\RR^d$.
  Let $\vec F = (F_1,\dots, F_N)$ be a (possibly improper) $N$-partition 
  of $\Omega$.
  For all $k=1,\dots,N$ with
  $\abs{F_k}>0$, let $x_k\in \Omega$, and $\rho_k>0$, be given so that
  the balls $B_{\rho_k}(x_k)$ are contained in $\Omega$, 
  are pairwise disjoint
  and $\abs{F_k\cap B_{r}(x_k)}>\frac 1 2 \omega_d r^d$
  for all $r\le \rho_k$.
 
  Let $A=\bigcup_k B_{\rho_k}(x_k)$.
  Then for every $\vec a = (a_0,\dots, a_N)$
  with $a_k \ge -\abs{F_k\cap B_{\rho_k}(x_k)}$ when $\abs{F_k}>0$,
  $a_k\ge 0$ when $\abs{F_k}=0$, 
  and $\sum_{k=1}^N a_k = 0$,
  there exists a partition $\vec F' = (F'_1, \dots, F'_N)$
  such that for all $k=1,\dots, N$:
  \begin{enumerate}
    \item[(i)] $F_k' \triangle F_k \subset A$,
    \item[(ii)] $\abs{F'_k \cap A} = \abs{F_k \cap A} + a_k$,
    \item[(iii)] $P(F'_k, \bar A) \le P(F_k,\bar A) + C_1 \cdot 
      \displaystyle\sum_{j=1}^N \abs{a_j}^{1-\frac 1 d}$,
  \end{enumerate}
  with $C_1=C_1(d,N)$ not depending on $\vec F$.  
\end{theorem}

\begin{definition}
\label{def:millefoglie}
Let $N\in \NN$, $J\subset \ENCLOSE{1,\dots,N}$ and for all $i \in \NN$ let $\bar \Pi^i$ be closed subsets of $\RR^d$.
We say that the $N$-partition $\vec F = (\vec F^i)_{i\in \NN}$ of $(\RR^d)^\NN$ is a  $J$-isoperimetric $N$-partition in 
$X = \bigotimes_{i\in \NN}\bar \Pi^i\subset (\RR^d)^\NN$ if 
for all compact $K\subset \RR^d$ and for all $N$-partition $\vec G = (\vec G^i)_{i\in \NN}$
of $(\RR^d)^\NN$
such that 
\begin{gather}
\label{eq:clo-finite}
\ENCLOSE{i\in \NN\colon \vec F^i \neq \vec G^i} \text{ is finite},\\
\label{eq:clo-difference}
\vec F^i \triangle \vec G^i \subset K\cap \bar \Pi^i \qquad \text{for all $i\in\NN$}, \\
\label{eq:clo-measure}
\sum_i \abs{G_j^i \cap K\cap \bar \Pi^i} = \sum_i \abs{ F_j^i \cap K\cap \bar \Pi^i} \qquad \text{for all $j\in J$,}
\end{gather}
one has
\begin{equation}
\label{eq:clo-thesis}
\sum_i \Per(\vec F^i, K\cap \bar \Pi^i) \le \sum_i \Per(\vec G^i, K\cap \bar \Pi^i). 
\end{equation}
\end{definition}

The following result extends \cite[Theorem 2.13]{NovPaoTor25} to the case of conical 
boundary data given on closed balls invading the space.

\begin{theorem}[closure of isoperimetric partitions with boundary data]\label{th:teoclosure}
Let $\vec S$ be a Caccioppoli $N$-partition of $\RR^d$ 
which is a cone, i.e.\ $\vec S = \RR_+ \cdot \vec S_0$ with $\vec S_0=\vec S \cap \partial B_1$.
Suppose that
\begin{equation}\label{eq:stima_theta}
\theta \defeq \sup_{\vec x\in \RR^d} \sup_{\rho>0} \frac{\Per(\vec S, B_\rho(\vec x))}{\rho^{d-1}} < +\infty.
\end{equation}

Let $\vec E_n=((E_n)_1,\dots, (E_n)_N)$ be a sequence of $N$-partitions of $\RR^d$ which are $J$-isoperimetric in $\bar B_{R_n} \subset \RR^d$
for some $J\subset \ENCLOSE{1,\dots,N}$ and $R_n\to+\infty$, such that $\vec E_n\setminus \bar B_{R_n} = \vec S \setminus \bar B_{R_n}$.

For $i\in \NN$, suppose $\vec F^i$ are $N$-partitions of $\RR^d$, and $\vec x_n^i\in \RR^d$,
 satisfy
\begin{gather}
    \label{eq:c-infty2}
    \lim_{n\to+\infty} \abs{\vec x_n^i-\vec x_n^j} \to \infty \qquad \text{if $i\ne j$}, 
    \\
    \label{eq:c-converge2}
    \vec E_n-\vec x_n^i \to \vec F^i \qquad \text{in $L^1_\loc(\RR^d)$ for each $i\in \NN$}.
\end{gather}

Moreover assume that, for all $i\in \NN$, we have, as $n\to +\infty$:
\begin{equation}
\label{EP-eq:clo-balls}
B_{R_n} - \vec x^i_n \stackrel{L^1_\loc}\longrightarrow \Pi^i
\end{equation}
where $\Pi^i$ is either the whole space $\RR^d$ or it is an open affine half-space of $\RR^d$, and
\begin{equation}
\label{EP-eq:clo-cloni}
\vec S - \vec x^i_n \stackrel{L^1_\loc}\longrightarrow \vec T^i,
\end{equation} 
\begin{equation}
\label{eq:hausloc}
    S_k-\vec x^i_n\text{ \rm locally Hausdorff converge to } T^i_k,
\end{equation}
where $\vec T^i$ is a $N$-partition of $\RR^d$.
Moreover suppose that for all $i\in \NN$ and all open bounded $B\subset \RR^d$
\begin{gather}
    \Per(\vec T^i,\partial B)=0 \implies 
    \lim_{n\to +\infty} \Per(\vec S-\vec x_n^i, B) = \Per(\vec T^i,B) \\
    \label{eq:3093576}
    \Per(\vec T^i, B) = \H^{d-1}(B \cap \partial\vec T^i).
\end{gather}

Then $\vec F = (\vec F^i)_{i\in \NN}$ is a  $J$-isoperimetric $N$-partition in  $X = \bigotimes_{i\in \NN}\bar\Pi^i\subset (\RR^d)^\NN$
(see Definition~\ref{def:millefoglie}). 
In particular each $\vec F^i$ is a $J$-isoperimetric partition in $\bar \Pi^i$ for all $i\in \NN$.
\end{theorem}

\begin{proof}
Note that if $\abs{\vec x_n^i}-R_n\to +\infty$ 
we would have $\Pi^i=\emptyset$, and this case has been excluded in the statement because it is not relevant.
If $\abs{\vec x_n^i}-R_n\to -\infty$ we have $\Pi^i=\RR^d$:
in this case the construction below is easier.
So, our focus is in the case when $\Pi^i$ is a half-space of $\RR^d$.

In the case when $\Pi^i$ is a half-space with outer normal $\nu_i$,
by \eqref{eq:c-converge2} and~\eqref{EP-eq:clo-balls} 
passing to the limit in $\vec E_n\setminus \bar B_{R_n}=\vec S\setminus \bar B_{R_n}$ we obtain that $\vec F^i \setminus \bar \Pi^i = \vec T^i \setminus \bar \Pi^i$.
Since $B_{R_n}(-\vec x_n^i)\to \Pi^i$ in $L^1_\loc$, one can check that
\[
    \nu^i = \lim_n \frac{\vec x_n^i}{\abs{\vec x_n^i}}.
\]
On the other hand, given $\vec x\in \RR^d$ one has  $\frac{ \vec x+\vec x^i_n}{\abs{ \vec x+\vec x^i_n}}\to \nu^i$ locally uniformly
as $n\to +\infty$, so that 
since $\vec S - \vec x_n^i$ converges locally to $\vec T^i$, being $\vec S$ a cone implies that $\vec T^i$ 
is equivalent to a cylinder perpendicular to $\partial \Pi^i$:
\[
    \partial \vec T^i = \partial \vec T^i_0 + \RR \nu^i    
\]
where $\vec T_0$ is the partition on $\partial \Pi^i$ which is the trace of $\vec T$. 
In particular, $\Per(\vec T^i, \partial \Pi^i+t\nu^i)=0$ for all $t\in \RR$.
From~\eqref{eq:stima_theta} we obtain
\[
  \Per (\vec T^i,B_\rho(\vec x)) \le \theta \rho^{d-1}, \qquad \forall \vec x\in \RR^d, \rho>0
\]
whence, using \eqref{eq:3093576}, considering $\partial \vec T_0$ a partition in dimension $d-1$,
\begin{equation}
    \label{eq:01455}
    \H^{d-2}(\partial \vec T_0 \cap B_\rho(\vec x)) \le \theta \rho^{d-2}, \qquad \forall \vec x \in \partial \Pi^i, \forall \rho>0.
\end{equation}

Given $\vec x\in B_R \cap \partial B_{R_n}(-\vec x_n^i)$, 
the normal vector to $B_{R_n}(-\vec x_n^i)$ in $\vec x$ is $\frac {\vec x+\vec x_n^i}{\abs{\vec x+\vec x_n^i}}$ 
which is uniformly converging, in $B_R$, to $\nu^i$ as $n\to +\infty$.
So, for $n$ large enough, we can assume that the angle between the normal to $\partial B_{R_n}$ and $\nu_i$ is small enough 
so that the orthogonal projection $\pi \colon \partial B_{R_n}(-\vec x_n^i) \to \partial \Pi^i$, if restricted 
to $B_R$, is injective and its inverse function is $2$-Lipschitz.
This means that for all measurable sets $X$ which are cylinders in direction $\nu^i$, one has
\begin{equation}\label{eq:448909}
    \H^{d-1}(X \cap B_R \cap \partial B_{R_n}(-\vec x_n^i))
    \le 2\H^{d-1}(X \cap B_R \cap \partial \Pi^i).
\end{equation}

Let $K\subset \RR^d$ be a compact set, 
let $\vec G = (\vec G^i)_{i\in \NN}$ be a partition satisfying~\eqref{eq:clo-finite}, \eqref{eq:clo-difference}, \eqref{eq:clo-measure}, 
and let $\eps>0$ be small enough so that if $C_1=C_1(d,N)$ is the constant given by Theorem~\ref{th:volume_fixing_variations_alternative}, we have
\begin{equation}
\label{EP-eq:assumption-epsilon}
    \eps < 1, \qquad \eps < (N C_1)^2.
\end{equation}

\begin{figure}
\begin{center}
\includegraphics[width=\textwidth]{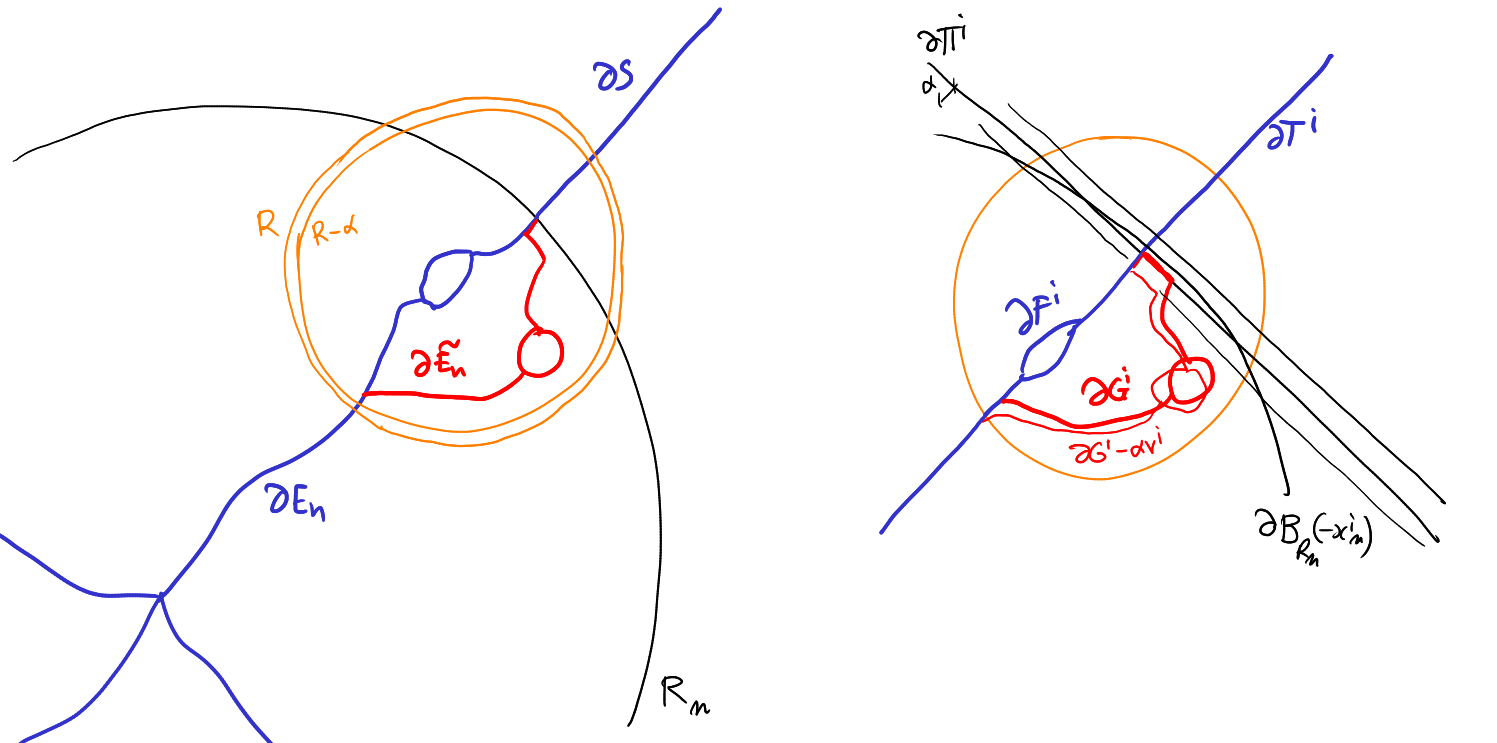}
\end{center}
\caption{A reference picture for the proof of Theorem~\ref{th:teoclosure}.}
\end{figure}

Our aim is to prove that the inequality~\eqref{eq:clo-thesis} is valid 
up to an error of order $\eps$.
This will be obtained by building a cluster $\hat {\vec E}_n$ competing to the minimality of $\vec E_n$ for some large $n$.
First a cluster $\tilde {\vec E}_n$ is obtained by patching all the given clusters $\vec G^i$ into large holes $\Omega_n^i = B_{R-\rho}(\vec x_n^i)\setminus B_{R_n}$ cut from $\vec E_n$.
The radius $R>0$ is large enough to contain all the variations $\vec G^i\triangle \vec F^i$ while $n$ will be taken large enough 
so that the balls $B_R(\vec x_n^i)$ will be disjoint. 
However these balls in general can contain large portions of the boundary $\partial B_{R_n}$ where we are required to keep 
the original boundary data of our sequence $\vec E_n$.
To facilitate the matching of the boundary, the cluster ${\vec G}^i$ will be slightly translated inside the container $\Pi^i$
so that the unknown behaviour of $\vec G^i$ on the boundary is replaced by $\vec T^i$ which can be assumed to be controlled.
The parameter $\alpha$ will measure the entity of the translation. 
By taking $n$ large enough we will ensure that 
$\alpha$ is larger than the distance, in $B_R$, between the translation of the curved boundary $\partial B_{R_n}(-\vec x_n^i)$
and its limit $\partial \Pi^i$.
The cutting radius $R$ must be slightly adjusted by a small quantity $\rho$ to assure we don't have perimeter concentrating 
on $\partial B_R$. 
A small amount of the mass of the regions with prescribed measure of $\vec E_n$ can remain outside all the balls $B_{R_n}$, 
we are hence forced to slightly adjust the measure of the regions of $\vec G^i$, using the volume fixing Lemma, 
to recover the constraint and obtain a competitor 
$\hat{\vec E}_n$ to $\vec E_n$. 
In preparation to this we have to fix the small balls $B_{r^i}(\vec y_k^i)$ inside the regions of $\bar G^i$
where the mass will be exchanged as needed, and we must take care of ensuring that these balls remain in the patching area.
The parameters $\eps, R, \alpha, n, \rho$ will be chosen in this order. 
All their properties will be declared at once when they are introduced, 
to make it clear that there is no circular dependency among them.

Let $I=\ENCLOSE{i\in \NN\colon \vec F^i\neq \vec G^i}$, $I$ is a finite set in view of~\eqref{eq:clo-finite}.
Let $R>0$ be large enough so that $K\subset B_{R-1}$, and $\Per (\vec G^i, \partial B_R))$.

In view of applying Theorem~\ref{th:volume_fixing_variations_alternative} to $G^i$,
for every $i\in I$ consider points $\vec y_1^i, \dots, \vec y_{N}^i$ and a radius $r_i>0$
such that $B_{r_i}(\vec y_k^i)\Subset \Pi^i\cap B_{R-1}$, and,
\[
    \abs{G^i_k\cap \Pi^i \cap B_{R-1}} > 0 
    \implies 
    \forall \rho<r_i\colon \abs{G^i_k\cap B_\rho(\vec y_k^i)} > \frac 1 2 \omega_d \rho^d.
\]
This can be done by taking $\vec y^i$ a point of density $1$ in $G^i_k\cap \Pi^i\cap B_{R-1}$ and $r_i$ small enough.

For all $\Pi^i$ which are half-spaces, let $\nu^i$ be the normal exterior versor, and set $\nu^i\defeq 0$ when $\Pi^i = \RR^d$.
Let $\alpha\in(0,\frac 1 2)$ be sufficiently small so that, being $C_1=C_1(d, N)$ as in Theorem  ~\ref{th:volume_fixing_variations_alternative}, we have
\begin{gather}
    \label{EP-eq:fik33}
    \sup_{\rho<\alpha} \Abs{\vec m(\vec G^i \triangle (\vec G^i -\rho \nu^i))}_1 \le \frac 1 2 \enclose{\frac{\eps}{N^2C_1}}^{\frac d {d-1}} \le \eps, \\
    \label{EP-eq:fik34}
    \sup_{\rho<\alpha} \Abs{\vec m(\vec F^i \triangle (\vec F^i -\rho \nu^i))}_1 \le \eps, \\
    \label{eq:fik39}
    \Per(\vec G^i,B_{R+\alpha}\cap (\Pi^i+2\alpha\nu^i)) - \Per(\vec G^i, B_{R-\alpha}\cap(\Pi^i + 2\alpha\nu^i) \le \eps,\\
    \label{eq:fik35}
    B_{r_i}(\vec y_k^i)\Subset \Pi^i-\alpha \nu^i,\\
    \label{EP-eq:fik36}
    \Per(\vec T^i,B_R\cap (\bar \Pi^i+2\alpha\nu^i)\setminus(\bar \Pi^i-2\alpha\nu^i)) \le 4 \alpha \theta R^{d-2} \le \eps,\\
    \label{eq:fik38}
    \Per(\vec T^i,B_R\setminus (\Pi^i+2\alpha\nu^i)) \le \Per(\vec T^i,B_{R-\alpha}\setminus (\Pi^i+2\alpha\nu^i)) + \eps,\\
    \label{eq:449367}
    \Per(\vec E_n, B_R(\vec x_n^i)) \le \Per(\vec E_n, B_{R-\alpha}(\vec x_n^i))+\eps.
\end{gather}

Take $n$ large enough so that, 
using \eqref{eq:c-infty2}, \eqref{eq:c-converge2},
and~\eqref{EP-eq:clo-balls},
we have
\begin{gather}
    R_n > R,\quad  \abs{\vec x_{n'}^i-\vec x_{n'}^\ell}> 3 R \qquad \text{if $i,\ell\in I$, $i\neq \ell$, $n'\ge n$},\\
    \label{EP-eq:fj32}
    \Abs{\vec m(((\vec E_n-\vec x_n^i)\triangle \vec F^i) \cap B_R)} < \frac 1 2 \enclose{\frac{\eps}{N^2 C_1}}^{\frac d {d-1}} \qquad \forall i\in I,\\
    \label{EP-grossi}
    (\bar \Pi^i -\alpha \nu^i) \cap B_{R} \subset B_{R_n}(-\vec x_n^i) \cap B_{R}, \\
    \label{EP-grossi2}
    \bar B_{R_n}(-\vec x_n^i) 
    \subset(\Pi^i +\alpha \nu^i),\\
    \label{EP-semicon}
       \Per(\vec F^i, B_{R}) \le \Per(\vec E_n, B_{R}(\vec x^i_n)) + \eps,\\
    \label{EP-cont2}
        \Per(\vec S-\vec x_n^i,B_R\setminus (\bar \Pi^i-\alpha\nu^i)) \le \Per(\vec T^i,B_R\setminus (\bar \Pi^i-\alpha\nu^i)) + \eps,\\
    \label{eq:clo-delta}
    B_R \cap (S_k - \vec x_n^i)\subset (T^i_k)_\delta, \qquad \text{with $\delta = {\frac{\eps}{4\theta R^{d-2}}}$.}
\end{gather}

For some fixed $\rho\in (0, \alpha)$ consider 
$\Omega_n^i\defeq B_{R-\rho}(\vec x_n^i) \cap B_{R_n}$,
define $\Omega_n \defeq \bigcup_i \Omega_n^i$, and 
\begin{equation}
    \tilde {\vec E}_n 
        \defeq \Enclose{\vec E_n \setminus \Omega_n} 
        \cup 
        \bigcup_{i\in I} \Enclose{(\vec G^i - \alpha \nu^i + \vec x_n^i)\cap \Omega_n^i}.
\end{equation}
Clearly $\tilde{\vec E}_n \triangle \vec E_n \subset \bar \Omega_n \subset \bar B_{R_n}$
moreover by slightly adjusting $\rho=\rho_n\in(0,\alpha)$,
in view of Lemma~\ref{lm:glueing}, we can suppose that 
\begin{equation}
\label{EP-bordomega} 
\begin{aligned} 
\MoveEqLeft{\Per (\tilde{\vec E}_n, \partial B_{R-\rho}(\vec x_n^i)\cap B_{R_n}) }\\
 &\le \frac 1 2 \sum_{i\in I}\Abs{\vec m((\vec E_n-\vec x_n^i) \triangle (\vec G^i-\alpha \nu^i)\cap B_R\setminus B_{R-1})}_1\\
 & = \sum_{i\in I}\Abs{\vec m((\vec E_n-\vec x_n^i) \triangle (\vec F^i-\alpha \nu^i)\cap B_R\setminus B_{R-1})}_1\\
 & \le \sum_{i\in I}\Abs{\vec m((\vec E_n-\vec x_n^i) \triangle (\vec F^i-\alpha \nu^i)\cap B_R)}_1\\
 & \le \sum_{i\in I}\enclose{\Abs{\vec m((\vec E_n-\vec x_n^i) \triangle \vec F^i\cap B_R)}_1
  + \Abs{\vec m(\vec F^i \triangle (\vec F^i-\alpha \nu^i)\cap B_R)}_1}\\
 & \le 2\eps     \qquad\text{[by \eqref{EP-eq:fj32} and \eqref{EP-eq:fik34}]}.\\
\end{aligned}
\end{equation}

For each $i\in I$, let us now consider the partition $\hat {\vec G}^i$ obtained applying Theorem~\ref{th:volume_fixing_variations_alternative}
to $\vec G^i$ on the balls $B_{\rho}(\vec y_k^i)$ so that, by~\eqref{eq:fik35},
\begin{gather}
\label{eq:44678}
  \hat {\vec G}^i \triangle \vec G^i \Subset B_{R-1} \cap \Pi^i 
    \subset B_{R-\alpha}\cap \Pi^i \subset (\Omega_n^i -\vec x_n^i)+\alpha\nu^i, \\
 \abs {(\hat G^i_j - \alpha \nu^i + \vec x_n^i) \cap \Omega_n^i} = \abs{(E_n)_j \cap \Omega_n} \qquad \forall i\in I\quad \forall j\in J.
\end{gather}
This is achieved by means of variations which changes the volume of each region $G^i_j$, with $j\in J$, 
by a (possibly negative) amount $a^i_j$ with, 
using~\eqref{EP-eq:fik33} and~\eqref{EP-eq:fj32},
\begin{align*}
\abs{a^i_j} 
    &\le \abs{(G^i_j-\alpha \nu^i)\triangle G^i_j\cap B_R} 
    + \abs{G^i_j\triangle F^i_j\cap B_R}
    + \abs{F^i_j\triangle((E_n)_j - \vec x_n^i)\cap B_R}\\
    &\le 
    \frac 1 2 \enclose{\frac{\eps}{N^2 C_1}}^{\frac d {d-1}}
    + 0 
    + \frac 1 2 \enclose{\frac{\eps}{N^2 C_1}}^{\frac d {d-1}}
    = \enclose{\frac{\eps}{N^2 C_1}}^{\frac d {d-1}}
\end{align*}
for $j\in J$, and $a^i_k=0$ for $k\not \in J$
(see also~\cite[Theorem 2.13]{NovPaoTor25} where this step is explained in detail).
Using also \eqref{EP-eq:fik33}, we have
\begin{equation}\label{EP-eq:4824}
      \Per(\hat {\vec G}^i, B_{R-\rho}) - \Per(\vec G^i, B_{R-\rho})
        \le  NC_1 \sum_{j=1}^N \abs{a^i_j}^{\frac {d-1}{d}} \le \eps.
\end{equation}
Define
\begin{equation*}
    \hat {\vec E}_n 
        \defeq \Enclose{\vec E_n \setminus \Omega_n} 
        \cup 
        \bigcup_{i\in I} \Enclose{(\hat{\vec G}^i - \alpha \nu^i + \vec x_n^i)\cap \Omega_n^i}
\end{equation*}
so that 
$\hat {\vec E}_n$ is a competitor to the constrained minimality of $\vec E_n$ in $\bar B_{R_n}$
\[
  \Per(\vec E_n, \bar B_{R_n}) \le \Per(\hat{\vec E}_n, \bar B_{R_n})
\]
and since $\hat {\vec E}_n \triangle \vec E_n\subset \bar B_{R_n} \cap \bigcup_{i\in I} \bar B_{R-\rho}(\vec x_n^i)$
we have
\begin{equation}
\label{EP-eq:4u8367}
  \sum_{i\in I} \Per(\vec E_n, \bar B_{R-\rho}(\vec x_n^i))
  \le \sum_{i\in I} \Per(\hat{\vec E}_n, \bar B_{R-\rho}(\vec x_n^i)).
\end{equation}
Let 
\[
\vec H^i \defeq (\vec S-\vec x_n^i)\setminus B_{R_n} \cup (\vec T^i \cap B_{R_n}).
\]
We need to estimate  $\Per(\vec H^i,B_R(\vec x_n^i)\cap \partial B_{R_n})$.
We assume $\abs{\vec x_n^i}\to +\infty$ otherwise there is nothing to prove.

By~\eqref{eq:clo-delta}, 
if we take 
$\delta = \frac{\eps}{4\theta R^{d-2}}$
we have
\[
    (\vec T^i\triangle (\vec S-\vec x_n^i))\cap B_R \subset (\partial \vec T^i)_\delta
\]
and hence
\begin{equation}
\label{eq:43444}
\begin{aligned}
    \MoveEqLeft{\Per(\vec H^i,B_R\cap \partial B_{R_n}(-\vec x_n^i))}\\
    &= \frac 1 2 \sum_{k=1}^N \H^{d-1}(((S_k-\vec x_n^i)\triangle T^i_k)\cap B_R \cap \partial B_{R_n}(-\vec x_n^i))\\
    &\le \H^{d-1}((\partial \vec T^i)_\delta \cap B_R \cap \partial B_{R_n}(-\vec x_n^i))\\
    &\le 2\H^{d-1}((\partial \vec T_0^i)_\delta \cap B_R \cap \partial \Pi^i) \qquad\text{[by \eqref{eq:448909}]}\\
    &\le 4\theta R^{d-2} \delta  \le \eps \qquad\text{[by \eqref{eq:01455}]}.
\end{aligned}
\end{equation}
We have
\begin{align*}
  \MoveEqLeft{\Per(\tilde{\vec E}_n, \bar B_{R-\rho}(\vec x_n^i)\setminus B_{R_n}) }\\
  &= \Per(\tilde{\vec E}_n-\vec x_n^i, \bar B_{R-\rho}\setminus B_{R_n}(-\vec x_n^i)) \\
  &= \Per(\vec H^i, \bar B_{R-\rho}\setminus B_{R_n}(-\vec x_n^i))\\
  &= \Per(\vec S-\vec x_n^i, \bar B_{R-\rho}\setminus \bar B_{R_n}(-\vec x_n^i)) \\
   & \quad + \Per(\vec H^i, B_{R-\rho} \cap \partial B_{R_n}(-\vec x_n^i)) \\
  &\le \Per(\vec S-\vec x_n^i, B_R \setminus (\bar \Pi^i-\alpha \nu^i)) + \eps \qquad\text{[by \eqref{EP-grossi} and \eqref{eq:43444}]} \\
  &\le \Per(\vec T^i, B_R \setminus (\bar \Pi^i-\alpha \nu^i)) + 2\eps \qquad\text{[by \eqref{EP-cont2}]} \\
  &\le \Per(\vec T^i, B_R \setminus (\Pi^i + 2\alpha \nu^i)) + 3\eps \qquad\text{[by \eqref{EP-eq:fik36}]}\\
  &\le \Per(\vec T^i, B_{R-\rho} \setminus (\Pi^i+2\alpha\nu^i)) + 4\eps \qquad \text{[by \eqref{eq:fik38}]}\\ 
  &= \Per(\vec G^i, \bar B_{R-\rho}\setminus (\Pi^i+2\alpha \nu^i)) + 4\eps\\
\end{align*}
while
\begin{align*}
 \Per(\tilde{\vec E}_n, \bar B_{R-\rho}(\vec x_n^i)\cap B_{R_n})
  &= \Per(\vec G^i-\alpha \nu^i,  B_{R-\rho}\cap B_{R_n}(-\vec x_n^i)) \\
  &\quad + \Per(\tilde{\vec E}_n,\partial B_{R-\rho}(\vec x^i_n)\cap B_{R_n})\\
  &\le \Per(\vec G^i, (B_{R-\rho}\cap B_{R_n}(-\vec x_n^i)) + \alpha \nu^i) + 2\eps\quad\text{[by \eqref{EP-bordomega}]}\\
  &\le \Per(\vec G^i, B_{R+\alpha}\cap(\Pi^i+2\alpha\nu^i)) +2\eps \qquad\text{[by \eqref{EP-grossi2}]}\\
  &\le \Per(\vec G^i, B_{R-\alpha}\cap(\Pi^i + 2\alpha\nu^i)) + 3\eps \qquad\text{[by \eqref{eq:fik39}]}
\end{align*}
so, recalling that 
$(\hat{\vec E}_n\triangle\tilde{\vec E}_n)\cap B_R(\vec x_n^i)
\subset \vec x_n^i -\alpha\nu^i+ (\hat{\vec G}^i \triangle \vec G^i) 
\Subset B_{R-1}(\vec x_n^i)-\alpha \nu^i
\subset B_{R-1+\alpha}(\vec x_n^i)\subseteq B_{R-\alpha}(\vec x_n^i)\Subset B_{R-\rho}$
\begin{equation}
\label{eq:30893}
    \begin{aligned}
      \Per(\hat {\vec E}_n, \bar B_{R-\rho}(\vec x_n^i)) 
      &\le \Per(\tilde {\vec E}_n, \bar B_{R-\rho}(\vec x_n^i)) + \eps \qquad\text{[by \eqref{EP-eq:4824}]}\\
      &= \Per(\tilde{\vec E}_n, \bar B_{R-\rho}(\vec x_n^i)\setminus B_{R_n}) \\
      &\quad + \Per(\tilde{\vec E}_n, \bar B_{R-\rho}(\vec x_n^i)\cap B_{R_n}) + \eps \\
      &\le \Per(\vec G^i, \bar B_{R-\rho}) + 8\eps
    \end{aligned}
\end{equation}
and the proof is concluded:
\begin{align*}
 \sum_{i\in I} \Per(\vec F^i,B_R) 
  &\le \sum_{i\in I} \Enclose{\Per(\vec E_n, B_R(\vec x_n^i)) + \eps} \qquad\text{[by~\eqref{EP-semicon}]}\\
  &\le \sum_{i\in I} \Enclose{\Per(\vec E_n, \bar B_{R-\rho}(\vec x_n^i)) + 2\eps} \qquad\text{[by~\eqref{eq:449367}]}\\
  &\le \sum_{i\in I} \Enclose{\Per(\hat{\vec E}_n, \bar B_{R-\rho}(\vec x_n^i)) + 2\eps} \qquad\text{[by~\eqref{EP-eq:4u8367}]}\\
  &\le \sum_{i\in I} \Enclose{\Per(\vec G^i, \bar B_{R-\rho}) + 10\eps} \qquad\text{[by \eqref{eq:30893}]}\\
  &\le \sum_{i\in I} \Enclose{\Per(\vec G^i, B_R) +10\eps}
\end{align*}
recalling that $\vec F^i \triangle \vec G^i \subset K \Subset B_{R-1}$, and that $I$ is finite.
\end{proof}

\section{Monotonicity formula and blow-down partitions}\label{secmono}


The following result is proven in \cite[Theorem 1.9, 2.1 and Proposition~4.1]{MagNov22}, 
along the lines of Allard \cite{All72}.

\begin{theorem}[Asymptotic expansion of planelike stationary varifolds]\label{teoallard}
Let $E\subset \RR^d\setminus \bar B_1$ be a set of locally finite perimeter such that ${\bf v}(\partial^* E\setminus \bar B_1,1)$ is a   
stationary varifold in $\RR^d\setminus \bar B_1$. Suppose that 
\[
\lim_{\rho\to +\infty} \frac{\Per(E,B_\rho\setminus \bar B_1)}{\omega_{d-1}\rho^{d-1}}=1,
\]
and, for some $r_k\to+\infty$, 
\[
    \frac{E}{r_k} \to  \ENCLOSE{x\in \RR^d\colon x_d < 0}
    \qquad\text{in $L^1_\loc(\RR^d)$ as $k\to +\infty$.}
\]
Then there exists $f\colon \RR^{d-1}\setminus\bar B_1\to \RR$ such that 
$E = \ENCLOSE{(x,y)\in \RR^{d-1}\times\RR\colon f(x)< y}$, up to a negligible set. 
Moreover, if $d\ge 4$ the function $f$ has the asymptotic expansion
\begin{equation}\label{expansion}
f(x) = a + \frac{b}{|x|^{d-3}} + O\left( \frac{1}{|x|^{d-2}}\right)
\qquad \text{as $|x|\to +\infty$,}
\end{equation}
with $a,b\in \RR$.

In particular $\frac E r$ converges to a half-space as $r\to +\infty$, and the blow-down is unique.
\end{theorem}

We also recall a result from \cite[Corollary 4.6]{BroNov24}.

\begin{theorem}[Existence of blow-down partitions]
\label{th:blow-down}
Let $\vec F$ be a $N$-isoperimetric partition of $\RR^d$, and let $R_k\to +\infty$. 
Then, up to a non-relabeled subsequence, as $k\to\infty$ we have
    \begin{eqnarray*}
    &&\frac{\vec F}{R_k} \to \vec F_\infty \qquad \text{in $L^1_{loc}(\RR^d)$},
    \\
    && \frac{\Per(\vec F, B_{R_k})}{R_k^{d-1}} \to \Per(\vec F_\infty, B_1),
    \end{eqnarray*}
    where $\vec F_\infty$ is a (possibly improper) $N$-isoperimetric conical partition of $\RR^d$. 
\end{theorem}

We plan to extend the previous results to isoperimetric partitions in a half-space $\Pi$ of $\RR^d$,
with boundary given by a $d-2$-dimensional affine space contained in $\partial\Pi$.
We start by showing the following monotonicity property.

\begin{theorem}[Monotonicity formula for isoperimetric partitions in a cone]
\label{th:monotonicity-formula}
    Let $\Pi$ be an open cone in $\RR^d$ with vertex in $0$.
    Let $\vec F$ be a Caccioppoli $N$-partition of $\RR^d$ 
    such that each region of $\vec F \setminus \bar \Pi$ is also a cone with vertex in $0$ (boundary condition).

    Suppose that $\vec F$ is locally isoperimetric in $\bar \Pi$.
    Suppose there exists $R>0$ be such that $F^k\subset B_R$ if $\abs{F^k}<+\infty$.
    Then
    \[
        \rho \mapsto \frac{\Per(\vec F,B_\rho\setminus B_R)}{\rho^{d-1}} 
    \]
     is non-decreasing on $[R,+\infty)$, so that 
    \[
        \rho \mapsto \frac{\Per(\vec F,B_\rho\cap \bar \Pi)}{\rho^{d-1}} =
      \frac{\Per(\vec F,B_\rho\setminus B_R) + \Per (\vec F,B_R)}{\rho^{d-1}} -\Per (\vec F,B_1\setminus\bar \Pi)  
    \]
  has limit as $\rho \to +\infty$ .
\end{theorem}
\begin{proof}
Without loss of generality suppose that $\abs{F^k}<+\infty$ if and only if $k\le M$ for some $M<N$.
Let $u(\rho)\defeq \Per(\vec F, B_\rho)$.
By the coarea inequality we know that for a.e.\ $\rho>0$ one has 
\[
u'(\rho) \ge \H^{d-2}(\partial \vec F\cap \partial B_\rho).
\]
For any $\rho>R$ lets consider the cone with the same boundary of $\vec F$ on $\partial B_\rho$
which is a partition $\vec C_\rho=(C_\rho^1, \dots, C_\rho^N)$ such that 
$x\in C_\rho^k$ if and only if $\rho \frac{x}{\abs{x}}\in F_k$.
Notice that $C_\rho^k = \emptyset$ for $k\le M$.
For a.e.\ $\rho>R$ one has
\[
    \Per(\vec C_\rho,\bar B_\rho) = \frac{\rho}{d-1} \H^{d-2}(\partial \vec F \cap \partial B_\rho)
\]
and $\vec C_\rho \triangle \vec F \subset \bar \Pi$.
For $\rho>R$ we are going to build a competitor to $\vec F$ by taking $\vec F_\rho =(F_\rho^1, \dots, F_\rho^N)$ with
\[
    F_\rho^k \defeq \begin{cases}
    F^k& \text{if $1 \le k \le M $}\\
    (F^k\setminus B_\rho) \cup (C_\rho^k\cap B_\rho) \setminus \displaystyle \bigcup_{j=1}^M F_j& \text{if $M < k \le N$}
\end{cases}
\]
so that $\vec F_\rho \triangle \vec F \subset \bar \Pi$ and, for a.e.\ $\rho>R$,
\[
 \Per(\vec F_\rho, B_\rho) \le \frac{r}{d-1} \H^{d-2}(\partial \vec F \cap \partial B_\rho) + \Per(\vec F,B_R).
\]
Hence for a.e.\ $\rho>R$
\begin{align*}
u(\rho) 
    &= P(\vec F,B_\rho) 
    \le P(\vec F_\rho, B_\rho)
    \le \frac{\rho}{d-1} \H^{d-2} (\partial \vec F \cap \partial B_\rho) + \Per(\vec F, B_R)\\
    &\le \frac{\rho}{d-1} u'(\rho) + \Per(\vec F, B_R).
\end{align*}
This is equivalent to
\[
 u'(\rho)\cdot  \rho^{1-d} - (d-1)\cdot u(\rho)\cdot \rho^{-d} + (d-1)\cdot\Per(\vec F, B_R) \cdot \rho^{-d}\ge 0
\]
or
\[
\frac{d}{d\rho} \Enclose{u(\rho)\cdot \rho^{1-d} - \Per(\vec F, B_R) \cdot \rho^{1-d}}\ge 0
\]
which means that 
\[
\rho \mapsto \frac{\Per(\vec F,B_\rho) - \Per(\vec F, B_R)}{\rho^{d-1}} 
\]
is non-decreasing.
Noting that $\vec F$ is a cone outside $\bar \Pi$, 
the function 
\[
\rho \mapsto \frac{\Per(\vec F,B_\rho\setminus \bar \Pi)}{\rho^{d-1}}
\] 
is constant, and the statement follows.
\end{proof}

The following result readily follows from Theorem \ref{th:monotonicity-formula}.
The proofs remain the same as in \cite[Corollary 4.6]{BroNov24}.

\begin{corollary}[Existence of blow-down partitions in $\Pi$]\label{corinfty}
    Let $\vec F$ and $\Pi$ be as in Theorem \ref{th:monotonicity-formula}, then there exists, and it is finite,
    the limit
    \[
    \Theta_\infty^\Pi(\vec F)\defeq \lim_{r\to +\infty}\frac{\Per(\vec F, B_r\cap \Pi)}{r^{d-1}}.
    \]
    Moreover, let $r_k\to +\infty$. Then, up to a non-relabeled subsequence,
    as $k\to\infty$ we have
    \begin{gather*}
    \frac{\vec F}{r_k} \to \vec F_\infty \qquad \text{in $L^1_{loc}(\RR^d)$},
    \\
    \Per(\vec F_\infty, B_1\cap\bar\Pi)
    = \Theta_\infty^\Pi(\vec F)
    \end{gather*}
    where $\vec F_\infty$ is a conical partition, locally minimal in $\bar\Pi$.
\end{corollary}

\section{Existence of non-standard isoperimetric partitions}\label{secexist}

The following definition provides a way to measure the perimeter that can be gained by removing a large portion of a 
partition $\vec E$ and replacing it with its blow-down partition.

\begin{definition}[defect]\label{def:defect}
Let $\vec E$ be a Caccioppoli $3$-partition of $\RR^d$ with $\abs{E_1}<+\infty$
and suppose that $\vec E$ has a unique blow-down $\vec E_\infty$. 
Let $\Pi\subset \RR^d$ be a positive cone with vertex in $0\in \RR^d$.
We define the \emph{defect} of $\vec E$ in $\bar \Pi$ as 
\begin{equation}
  \Delta^{\bar\Pi}_{\vec E} \defeq
  \liminf_{r\to +\infty} 
  \frac{\displaystyle 
  \Per(\vec E,\bar\Pi\cap B_r) -
  \Per(\vec E_\infty,\bar\Pi \cap B_r)
  }
  {\abs{E_1}^{\frac{d-1}{d}}}.
\end{equation}
In the case when $\Pi=\RR^d$ we simply write:
\[
 \Delta_{\vec E} \defeq \Delta_{\vec E}^{\RR^d}.
\]
\end{definition}

\begin{remark}\label{rem:defect}
Notice that, if $\vec E = (E, \emptyset, \RR^d\setminus E)$ is the $3$-partition associated to a set $E$ with $\abs{E}<+\infty$, then 
\[
  \Delta_{\vec E} = \frac{\Per(E)}{\abs{E}^{\frac{d-1} d}}
\]
is the usual \emph{isoperimetric ratio} of $E$.
\end{remark}

\begin{definition}[lens partition]
\label{def:lens-partition}
We say that a $3$-partition $\vec L = (L^1,L^2,L^3)$ of $\Omega\subset \RR^d$ is a lens partition in $\Omega$ if $L^1\subset \bar\Omega$ is a lens:
\[
 L^1 = B_R(p)\cap B_R(q), \qquad \text{with $p,q\in \RR^d$, $\abs{p-q}=R$}
\]
and $L^2 = H\setminus L^1$, $L^3=\RR^d\setminus (H \cup L^1)$
where $\partial H$ is the hyperplane equidistant from the points $p,q$:
\[
  H = \ENCLOSE{x\in \RR^d\colon \abs{x-p} \le \abs{x-q}}.
\]
\end{definition}

\begin{definition}[Simons' cone, barrel]
\label{def:simons-partition}
We call \emph{Simons' cone partition} the $3$-partition $\vec S=(S_1,S_2,S_3)$ of $\RR^8$
defined by
\begin{align*}
 S_1 &\defeq \emptyset,\\
 S_2 &\defeq \ENCLOSE{(\vec x,\vec y)\in \RR^4\times \RR^4\colon \abs{\vec x} \le \abs{\vec y}},\\
 S_3 &\defeq \RR^d\setminus S_2.
\end{align*}

We call \emph{barrel partition} the $3$-partition $\vec Q = (Q_1,Q_2,Q_3)$ of $\RR^8$
defined by 
\begin{align*}
  Q_1 &\defeq \BB^4\times\BB^4 =  \ENCLOSE{(\vec x, \vec y)\in \RR^4\times \RR^4\colon \abs{\vec x}\le 1, \abs{\vec y} \le 1},\\
  Q_2 &\defeq S_2 \setminus Q_1,\\
  Q_3 &\defeq S_3 \setminus Q_1.
\end{align*}
\end{definition}

\begin{remark}
Clearly $\vec m(\vec S) = (0, +\infty, +\infty)$
and it is well known that $\vec S$ is an isoperimetric partition since $S_2$ is a locally minimal set as proved 
in~\cite{BomDeGGiu69} (see also \cite{DePPao08}).
\end{remark}

\begin{lemma}[defects of lens and barrel partitions]
\label{lm:defect-partition}
Let $\vec L$ be a lens partition,
then
\begin{equation}
    \Delta_{\vec L} = 4 \sqrt[8]{4 \pi^3\Enclose{\frac {16} 9 - \frac{93}{35}\sqrt 3}}  \approx 7.29.
\end{equation}
Let $\vec Q$ be the barrel partition, then
\[
    \Delta_{\vec Q} = \sqrt[4]{8}\sqrt \pi \Enclose{4-\frac 8 7 \sqrt 2} \approx 7.10.
\]
\end{lemma}

\begin{proof}
We have
\begin{align*}
 \omega_d \defeq \abs{\BB^d} 
 &=  \frac{\pi^{\frac d 2}}{\frac d 2!}
 = \begin{cases}
    \frac{\pi^k}{k!} & \text{if $d=2k$ is even},\smallskip\\
    \frac{\pi^k 2^k}{d!!} & \text{if $d=2k+1$ is odd},
 \end{cases}\\
 \H^d(\SS^d) &= (d+1)\omega_{d+1}
\end{align*}
whence
\[
\abs{\BB^4} = \frac{\pi^2}{2}, 
\qquad
\H^3(\SS^3) = 2 \pi^2,
\qquad
\abs{\BB^7} = \frac{16}{105} \pi^3,
\qquad
\H^6(\SS^6) = 7 \omega_7 = \frac{16}{15}\pi^3.
\]

Let $\vec S=(\emptyset, S_2, \RR^8\setminus S_2)$ be the Simons' cone partition and let $\vec Q = (Q_1,S_2\setminus Q_1,\RR^8\setminus(S_1\cup Q_1))$ be the \emph{barrel} partition (see Definition~\ref{def:simons-partition}). 
Then
\begin{align*}
\abs{Q_1} &= \abs{\BB^4}^2 = \omega_4^2 = \frac{\pi^4}{4},\\
\Per(Q_1) 
    &= 2\H^3(\SS^3)\H^4(\BB^4) 
    = 2\pi^4, \\
\Per(S_2,Q_1) 
    &= \H^3(\SS^3)^2\int_0^1 x^3x^3\sqrt 2 dx
    = \frac 4 7 \sqrt 2 \pi^4.
\end{align*}

Now consider the lens partition $\vec L=(L_1,H\setminus L_1,\RR^8 \setminus (H\cup L_1)$ with 
\begin{align*}
      H &= \ENCLOSE{x\in \RR^d\colon x_8\ge 0}\\
      L_1 &= (\BB^8 + \frac {e_8}{2})\cap (\BB^8 - \frac {e_8}{2})\\
\end{align*}
where $e_8=(0,\dots,0, 1)\in \RR^8$
(see Definition~\ref{def:lens-partition}).

One has
\begin{align*}
\abs{L_1} 
    &=  2 \int_{\frac 12}^1 \abs{\sqrt{1-y^2}\cdot \BB^7} \,dy
    =  2 \omega_7 \int_{\frac 12}^1 (1-y^2)^{\frac 7 2}\, dy\\
    &= 2 \omega_7 \int_0^{\frac \pi 3} \sin^8 \theta\, d\theta
    = \omega_7 \enclose{\frac{35}{192} \pi -\frac{279}{1024} \sqrt 3}\\
    &= \frac{\pi^4}{36} -\frac{93}{2240} \sqrt 3\pi^3,
\end{align*}
\[
  \Per(H,L_1) 
    = \abs{\frac{\sqrt 3} 2 \BB^7} 
    = \frac{27}{128}\sqrt 3\omega_7
    = \frac{9}{280}\sqrt 3\pi^3.
\]
If we let 
$V(r)=\abs{L_1\cap H}$,
on one hand we have 
\[
 V'(1) = \frac 1 2 \H^7(\partial L_1) - \frac 1 2 \H^7 \enclose{\frac{\sqrt 3}{2} \BB^7}
\]
and on the other hand we have
$2 V(r) = r^8 \abs{L_1}$
hence 
\[
  \Per(L_1) 
    = 8\abs{L_1} + \omega_7 \enclose{\frac{\sqrt 3} 2}^7
    = 8\abs{L_1} + \frac{16}{105}\cdot \frac{27}{128}\sqrt 3 \pi^3
    = \frac 2 9 \pi^4 - \frac 3{10} \sqrt 3 \pi^3
\]

In view of Remark~\ref{rem:defect}
we can finally compute and compare the defects:
\begin{align*}
    \Delta_{\vec Q} &=
    \frac{\Per(Q_1)-\Per(S_2,Q_1)}{\abs{Q_1}^{\frac 7 8}}
    = \frac
            {2\pi^4 - \frac 4 7 \sqrt 2 \pi^4}
            {\enclose{\frac{\pi^4}{4}}^{\frac 78}}
    \\
    &= \sqrt[4]{8} \sqrt{\pi}\Enclose{4 - \frac 8 7  \sqrt 2}  
    \approx 7.10, \\     
    \Delta_{\vec L} &=
    \frac{\Per(L_1)-\Per(H,L_1)}{\abs{L_1}^{\frac 7 8}}
    = \frac
        {\frac 2 9 \pi^4 - \frac{93}{280}\sqrt 3 \pi^3}
        {\enclose{{\frac{\pi^4}{36} -\frac{93}{2240} \sqrt 3\pi^3}}^{\frac 7 8}}
    \\
    &= 4\sqrt[8]{4\pi^3\Enclose{\frac{16}{9}\pi - \frac{93}{35} \sqrt 3} }
    \approx 7.29.
\end{align*} 
\end{proof}

\begin{lemma}\label{lemmamain}
Let $R_n>0$, $R_n\to+\infty$ be given.
Let $\vec S$ be the Simons' partition (see Definition~\ref{def:simons-partition}).
Let $\vec E_n=((E_n)_1,(E_n)_2,(E_n)_3)$ be a $3$-partition of $\RR^8$ such that $\vec E_n\setminus \bar B_{R_n} = \vec S\setminus \bar B_{R_n}$,
and $\abs{(E_n)_1}=1$.
Suppose  that for all $i\in \NN$, $i>0$,
there exist $\vec x_n^i\in \RR^d$, and $\vec F^i$, such that
\begin{gather*}
    \vec E_n -  \vec x_n^i \stackrel{L^1_\loc}\longrightarrow \vec F^i,\\
    \sum_{i=1}^{+\infty}\abs{F^i_1} 
    = 1,\\
    \lim_{n\to +\infty} \abs{ \vec x_n^i- \vec x_n^j}= +\infty \qquad\text{for all $i\neq j$.}    
\end{gather*}
Let $\Pi^i$ 
be the $L^1_\loc$-limit of the balls $B_{R_n}-\vec x_n^i$ as $n\to \infty$.
Clearly $\bar \Pi^i$ is either $\RR^8$ (if $R_n-\abs{\vec x_n^i}\to \infty$)
or a closed half-space.
Suppose
moreover that
there exist partitions 
$\vec F^i_\infty=(\emptyset, (F^i_\infty)_2, (F^i_\infty)_3)$,
$\vec T^i = (\emptyset, T^i_2, T^i_3)$, $\vec H^i = (\emptyset, H^i_2,H^i_3)$,  
with $\partial \vec H^i=\emptyset$ or $\partial \vec H^i$ is a hyperplane passing through the origin, so that
\begin{gather}
        \frac{\vec F^i}{r}  \stackrel{L^1_\loc}{\longrightarrow} \vec F^i_\infty \qquad \text{\rm as }\, r\to +\infty\\
        \vec S - \vec x_n^i \stackrel{L^1_\loc}\longrightarrow \vec T^i \quad \text{\rm as  }\, n\to +\infty\\
        (\vec F^i_\infty +\vec y^i) \cap \Pi^i = (\vec H^i+\vec y^i) \cap \Pi^i \\
        (\vec F^i_\infty +\vec y^i) \setminus \Pi^i = \vec T^i \setminus \Pi^i,
\end{gather}
for some $\vec y^i\in \RR^d$. 
Suppose also 
\begin{equation}
\Abs{\vec m ((\vec F^i\triangle (\vec H^i+\vec y^i))\cap (B_r\setminus B_{r-1})\cap \Pi^i)}_1 \to 0.
\end{equation}
Then
\[
\limsup_{n\to+\infty} \Delta_{\vec E_n} \ge \inf_i \Delta_{\vec F^i}^{\bar \Pi^i}.
\]
\end{lemma}

\begin{figure}
\begin{center}
\includegraphics[height=4cm]{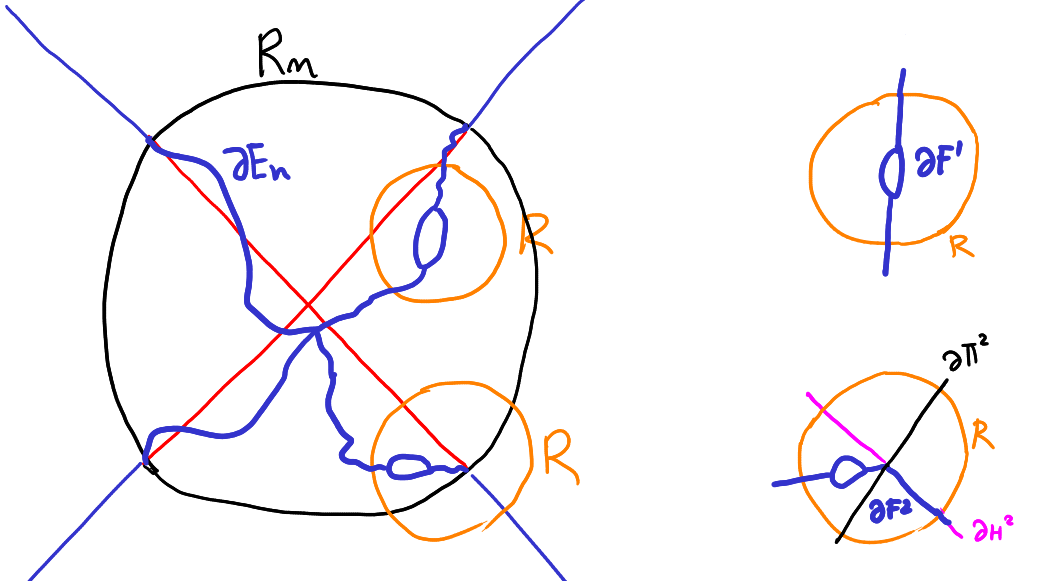}
\caption{A reference picture for the proof of Lemma~\ref{lemmamain}.}
\end{center}
\end{figure}

\begin{proof}

Fix $\eps\in(0,1)$. 
Take $N$ large enough such that
\begin{equation}\label{eq:39564}
 \sum_{i=N+1}^{+\infty} \abs{F^i_1} \le \sum_{i=N+1}^{+\infty} \abs{F^i_1}^{\frac{7}{8}} < \eps.
\end{equation}
Take also $r>1$ large enough so that
such that $\sum_{i=1}^N \abs{F^i_1\cap B_{r-1}} > 1 - 2\eps$ and
\begin{gather}
    \sum_{i=1}^N \Abs{\vec m ((\vec F^i\triangle (\vec H^i+\vec y^i))\cap (B_r\setminus B_{r-1})\cap \Pi^i)}_1 < \eps, 
        \nonumber\\
    \sum_{i=1}^N \left(\Per(\vec F^i, B_\rho\cap\bar\Pi^i) - \Per(\vec H^i+\vec y^i,B_\rho\cap\bar  \Pi^i)\right) \ge \left(\sum_{i=1}^N \Delta_{\vec F^i}^{\bar \Pi^i} \cdot \abs{F^i_1}^{\frac{7}{8}}\right) - \eps ,
    \label{eq:93566}
\end{gather}
for all $\rho \in [r-1,r]$.

Fix $n$ large enough so that the balls $B_{r}(\vec x_n^i)$ are pairwise disjoint, and 
\begin{gather*}
 \sum_{i=1}^N\Abs{\vec m(((\vec E_n- \vec x_n^i) \triangle \vec F^i) \cap B_r)}_1 < \eps,\\
 \sum_{i=1}^N\abs{((B_{R_n} - \vec x_n^i) \triangle \Pi^i)\cap B_r}<\eps \\ 
 \Per(\vec H^i+\vec y^i, B_r \cap \bar B_{R_n}(- \vec x^i_n)\setminus \bar \Pi^i) < \frac \eps N,
\end{gather*}
so that 
\begin{equation}\label{eqzero}
 \sum_{i=1}^N\Abs{\vec m((\vec E_n- \vec x_n^i) \triangle (\vec H^i+\vec y^i))\cap (B_r\setminus B_{r-1})\cap (B_{R_n}- \vec x_n^i)}_1 < 3\eps.
\end{equation}
As a consequence, for $\rho \in [r-1, r]$
\begin{align}
    \Per(\vec H^i+\vec y^i, B_\rho\cap \bar B_{R_n}(- \vec x_n^i)) 
    & \le  
    \Per(\vec H^i+\vec y^i, B_\rho\cap \bar\Pi^i) \\
    &+ \Per(\vec H^i+\vec y^i, B_\rho \cap \bar B_{R_n}(- \vec x^i_n)\setminus \bar \Pi^i)\nonumber\\
    & \le 
    \Per(\vec H^i+\vec y^i, B_\rho \cap \bar\Pi^i) + \frac \eps N.
    \label{eq:37612}
\end{align}
By the semicontinuity of perimeter and the convergence of $B_{R_n}(-\vec x_n^i)$ to $\Pi^i$,
we can also suppose that 
\begin{equation}\label{eqquattro}
 \Per(\vec F^i, B_r\cap \bar \Pi^i)
 \le
 \Per(\vec E_n, B_r( \vec x_n^i)\cap \bar B_{R_n}) + \frac \eps N,
\end{equation}
indeed, for $\delta>0$ we have
\begin{align*}
 \Per(\vec F^i, B_r\cap \bar \Pi^i)
 &\le
\liminf_n \Per(\vec E_n-\vec x_n^i, B_r\cap B_{R_n+\delta}(-\vec x_n^i))\\
&= \liminf_n \left[\Per(\vec E_n, B_r( \vec x_n^i)\cap \bar B_{R_n})\right.\\
&+ \left.\Per(\vec S, B_r( \vec x_n^i)\cap (B_{R_n+\delta}\setminus \bar B_{R_n}))\right]
\\
&\le \liminf_n \Per(\vec E_n, B_r( \vec x_n^i)\cap \bar B_{R_n}) + C \delta r^{d-1},
\end{align*}
which proves \eqref{eqquattro} by choosing $\delta$ such that $C \delta r^{d-1} < \eps/(2N)$ and $n$ large enough.

Let now
\begin{equation}\label{eq:44339}
    \vec G_n^i \defeq ((\vec x_n^i + \vec y^i + \vec H^i) \cap B_{R_n}) \cup (\vec S\setminus B_{R_n}).
\end{equation}
Notice that
\[
\vec G^i_n - \vec x^i_n \rightarrow ((\vec y^i+\vec H^i) \cap \Pi^i) \cup (\vec T^i \setminus \Pi^i) = \vec y^i + \vec F^i_\infty,
\]
where the limit is locally in the Hausdorff distance. Moreover, 
we have
\begin{equation}\label{eqg}
\limsup_n \Per(\vec G_n^i-\vec x^i,B_\rho\cap\bar B_{R_n}(-\vec x^i))\le \Per(\vec H^i+\vec y^i, B_\rho \cap \bar\Pi^i).
\end{equation}
Indeed, recalling from \eqref{eq:37612} for $\delta>0$ we have
\begin{align*}
\limsup_n \Per(\vec G_n^i-\vec x_n^i,B_\rho\cap\bar B_{R_n}(-\vec x_n^i))\le& 
\limsup_n \Per(\vec G_n^i-\vec x^i,B_\rho\cap B_{R_n+\delta}(-\vec x_n^i))\\ 
\le& 
\Per(\vec H^i+\vec y^i, B_\rho \cap (\Pi^i)_{2\delta}))\\
&+\Per(\vec T^i, B_\rho \cap ((\Pi^i)_{2\delta})\setminus \bar\Pi^i)),
\end{align*}
which gives \eqref{eqg} letting $\delta\to 0$. Hence, from \eqref{eqg},  for $n$ large enough it holds good 
\begin{equation}\label{eq:ahi}
   \sum_{i=1}^N\Per(\vec G_n^i-\vec x^i,B_\rho\cap\bar B_{R_n}(-\vec x^i))\le \sum_{i=1}^N \Per(\vec H^i+\vec y^i, B_\rho \cap \bar\Pi^i)+\eps.
\end{equation}


Let 
\[
  \hat {\vec E}_n  = (\emptyset, (E_n)_1 \cup (E_n)_2, (E_n)_3).
\]
Clearly for all Borel $B\subset \RR^d$ one has
\begin{equation}\label{eq:7498}
 \Per(\hat {\vec E}_n, B) \le \Per (\vec E_n,B).
\end{equation}
By construction we have
\begin{gather}
 \Abs{\vec m(\hat{\vec E}_n\triangle \vec G^i_n)\setminus B_{R_n} }_1 = 0 \\
 \sum_{i=1}^N \Abs{\vec m(\hat {\vec E}_n\triangle \vec E_n) \setminus B_{r-1}(\vec x^i_n)}_1 < 6\eps \\
 \sum_{i=1}^N \Abs{\vec m (\vec E_n \triangle (\vec x^i_n + \vec y^i+\vec H^i) \cap (B_r(\vec x^i_n))\setminus B_{r-1}(\vec x^i_n)\cap B_{R_n}}_1 < 3 \eps \\
 \Abs{\vec m(((\vec x^i_n+\vec y^i+\vec H^i)\triangle \vec G^i_n) \cap B_{R_n})}_1 = 0,
\end{gather}
hence
\begin{equation}\label{eq:8449}
 \sum_{i=1}^N \Abs{\vec m((\hat{\vec E}_n \triangle \vec G^i_n) \cap (B_r(\vec x^i_n)\setminus B_{r-1}(\vec x^i_n)))}_1 < 9\eps.
\end{equation}

For $\rho\in (r-1,r)$ we define
\begin{equation}
    \tilde {\vec E}_n \defeq 
  (\hat{\vec E}_n \setminus \bigcup_{i=1}^N B_\rho(\vec x^i_n)) \cup \bigcup_{i=1}^N (\vec G^i_n \cap B_\rho(\vec x^i_n)).
\label{eq:5592}
\end{equation}
By Lemma~\ref{lm:glueing}, thanks to~\eqref{eq:8449}, we can appropriately choose $\rho$, so that
\begin{equation}\label{eqtre}
 \sum_{i=1}^N \Per(\tilde {\vec E}, \partial B_{\rho}(\vec x^i_n)) < 5\eps.
\end{equation}
Let $\vec M_n$ be a partition 
minimizing the perimeter in the family of partitions $\vec F=(\emptyset,F_2,F_3)$ which coincide with $\hat {\vec E}_n$ on 
$\RR^d\setminus (B_{R_n}\cap \bigcup_{i=1}^N B_r(\vec x^i_n))$,
so that $\vec M_n \triangle \hat {\vec E}_n \subset \bar B_{R_n}$.
It follows that
\begin{equation}\label{equno}
  \Per (\vec M_n, \bar B_{R_n})\ge \Per(\vec S, \bar B_{R_n}).
\end{equation}
Notice that $\tilde {\vec E}_n$ coincides with $\hat {\vec E}_n$ outside the intersections $ B_{R_n}\cap B_R(\vec x^i_n)$,
hence, by minimality of $\vec M_n$, we have 
\begin{equation}\label{eqdue}
\Per(\tilde{\vec E}_n, \bar B_{R_n}) \ge \Per(\vec M_n, \bar B_{R_n}).
\end{equation}
It follows that
\begin{align*}
\Per(\vec S, \bar B_{R_n})
&\le \Per(\tilde {\vec E}_n, \bar B_{R_n})  \qquad \text{[by \eqref{equno} and~\eqref{eqdue}]}\\
&= \Per(\tilde {\vec E}_n, B_{R_n}) + \Per(\tilde {\vec E}_n, \partial B_{R_n})
\\
&= \Per(\hat {\vec E}_n, B_{R_n}\setminus\cup_{i=1}^N \bar B_{\rho}(\vec x^i_n)) 
+ \Per(\tilde{\vec E}_n, \partial B_{R_n})
\\
&\quad + \sum_{i=1}^N \Enclose{\Per(\vec G^i_n ,B_\rho(\vec x^i_n)\cap B_{R_n}) 
+ \Per(\tilde{\vec E}_n,(\partial B_{\rho}(\vec x^i_n))\cap B_{R_n})} \qquad\text{[by~\eqref{eq:5592}]}
\\
&\le \Per(\hat {\vec E}_n, B_{R_n}\setminus\cup_{i=1}^N \bar B_{\rho}(\vec x^i_n)) 
+ \Per( \hat{\vec E}_n, \partial B_{R_n})
\\
&\quad + \sum_{i=1}^N \Enclose{\Per(\vec G^i_n ,B_\rho(\vec x^i_n)\cap B_{R_n}) 
+ \Per(\tilde{\vec E}_n,\partial (B_{\rho}(\vec x^i_n)\cap B_{R_n}))}
\\
&\le \Per(\hat {\vec E}, B_{R_n}\setminus\cup_{i=1}^N \bar B_{\rho}(\vec x^i_n)) 
+ \Per( \hat{\vec E}_n, \partial B_{R_n})
\\
&\quad + \sum_{i=1}^N \Enclose{\Per(\vec G^i_n ,B_\rho(\vec x^i_n)\cap B_{R_n}) 
+\Per(\vec G^i_n,B_{\rho}(\vec x^i_n)\cap \partial B_{R_n})} + 5\eps \qquad\text{[by~ \eqref{eqtre}]}
\\
&\le \Per(\vec E_n, B_{R_n}\setminus\cup_{i=1}^N \bar B_{\rho}(\vec x^i_n)) 
+ \Per( \vec E_n, \partial B_{R_n})
\\
&\quad + \sum_{i=1}^N \Per(\vec G^i_n ,B_\rho(\vec x^i_n)\cap \bar B_{R_n})
+ 5 \eps 
\qquad\text{[by~\eqref{eq:7498}, \eqref{eq:44339} ]}
\\
&\le \Per(\vec E_n, \bar B_{R_n}) \\
&\quad + \sum_{i=1}^N \Enclose{\Per(\vec G^i_n ,B_\rho(\vec x^i_n)\cap \bar B_{R_n}) - \Per(\vec E_n, B_\rho(\vec x^i_n)\cap \bar B_{R_n})} + 6\eps\qquad \text{[by \eqref{eq:ahi}]} \\
&\le \Per(\vec E_n,\bar B_{R_n}) \qquad \text{[by \eqref{eqquattro}]} \\
&\quad + \sum_{i=1}^N \Enclose{\Per(\vec G^i_n ,B_\rho(\vec x^i_n)\cap \bar B_{R_n}) -\Per(\vec F^i, B_\rho\cap \bar\Pi^i)}+
6 \eps. 
\end{align*}
Hence we get
\begin{align*}
\Delta_{\vec E_n} &= \lefteqn{\Per(\vec E_n,\bar B_{R_n}) - \Per(\vec S,\bar B_{R_n})} \\
&\ge 
\sum_{i=1}^N \Enclose{\Per(\vec F^i, B_{\rho}\cap \bar\Pi^i)
 - \Per(\vec G^i_n, B_\rho(\vec x^i_n)\cap \bar B_{R_n})} 
 -6\eps\\
 &\ge \sum_{i=1}^N \Enclose{
 \Per(\vec F^i, B_{\rho}\cap \bar\Pi^i)
 -\Per(\vec H^i + \vec y^i, B_{\rho}\cap \bar \Pi^i)} -7\eps
    \qquad\text{[by~\eqref{eq:ahi}]}\\
 &\ge \sum_{i=1}^N \Delta_{\vec F^i}^{\bar \Pi^i} \cdot \abs{F^i_1}^{\frac{7}{8}} - 8\eps
 \qquad\text{[by~\eqref{eq:93566}]}\\
 &\ge \left(\inf_{i\in \{1,\ldots N\}} \Delta_{\vec F^i}^{\bar \Pi^i}\right) \left(\sum_{i=1}^{+\infty} \abs{F^i_1}^{\frac{7}{8}} -\eps\right) - 8\eps
 \qquad\text{[by~\eqref{eq:39564}]}\\
 &\ge \inf_i \Delta_{\vec F^i}^{\bar \Pi^i} (1-\eps) - 8\eps,
\end{align*}
which gives the thesis letting $\eps\to 0$.
\end{proof}

In the following theorem we show that, when $\Pi$ is a half-space, the blow-down partition in Corollary \ref{corinfty} 
is necessarily flat in $\Pi$.

\begin{theorem}[half-space]
\label{th:piano}
Let $E\subset \RR^d$ be a locally minimal set in $\RR^d$ and suppose there exists a half-space $\Pi\subset \RR^d$ 
such that $\abs{E\setminus \Pi}=0$.
Then $\partial E$ is a hyper-plane.
\end{theorem}

\begin{proof}

Let $E_\infty$ be a blow down of $E$ as given by Corollary~\ref{corinfty}. 
Clearly $\abs{E_\infty\setminus \Pi} = 0$ and $0\in \partial E_\infty \cap \partial \Pi$.
But $\Pi$ itself is minimal, hence we can apply the strong maximum principle (see \cite[Corollary 1]{Sim87})
to state that $\partial E_\infty = \partial \Pi$.
\end{proof}

\begin{theorem}\label{teoflat}

Let $\vec F= (F_1,F_2,F_3)$ be a $3$-partition of $\RR^d$ with $\abs{F_1}<+\infty$.

Assume that $\Pi$ is a half-space with $0\in \partial \Pi$
and suppose that each region of $\vec F \setminus \bar \Pi$ is a cone with vertex in $0$.
In particular $\abs{F_1\setminus \Pi}=0$.

Suppose that $\vec F$ is isoperimetric in $\bar \Pi$. 
Let $\vec F_\infty$ be a blow-down of $\vec F$, as in Corollary~\ref{corinfty}.

\begin{itemize}
    \item[(i)] If $\abs{F_2\setminus \Pi}=0$ or $\abs{F_3\setminus \Pi}=0$ then 
    either $\partial \vec F_\infty=\emptyset$ or $\partial \vec F_\infty = \partial \Pi$.
    \item[(ii)] If 
    $\vec F \setminus \bar \Pi = (\emptyset, H\setminus\bar\Pi, (\RR^d\setminus H)\setminus \bar \Pi)$,
    where $H$ is a half-space not parallel to $\partial \Pi$ with $0\in \partial H$
    then $\partial \vec F_\infty \cap \bar \Pi$ is a half-hyperplane.
    \end{itemize}
\end{theorem}

\begin{proof}\,
\begin{itemize}
    \item[(i)] Without loss of generality we can assume that $\abs{F_3\setminus \Pi}=0$. 
    Then $\abs{F_\infty^3 \setminus \Pi} = 0$ and $\abs{F_\infty^1}=0$. 
    Since $\vec F_\infty$ is isoperimetric in $\bar \Pi$ by Theorem \ref{th:teoclosure},
    $\abs{F_3\setminus \Pi}=0$ and $\Pi$ is convex, then $\partial \vec F_\infty$ is locally minimal in $\RR^d$,
    and by Theorem~\ref{th:piano} it is either empty or a hyperplane. In the latter case it necessarily coincides with $\partial\Pi$.
    \item[(ii)]
    By Corollary \ref{corinfty}, 
    $\vec F_\infty$ is an isoperimetric conical partition in $\bar\Pi$.
    Since $\abs{F_\infty^1}=0$ we have $\partial \vec F_\infty = \partial F_\infty^2$.
    Consider the $(d-1)$-dimensional rectifiable current $T$ representing the restriction of $\partial \vec F_\infty$
    in $\bar\Pi$ (i.e.\ the current representing the integration on the rectifiable set $\partial F_\infty^2\cap \bar\Pi$ with orientation 
    given by the normal vector $\nu_{F_\infty^2}$).
    Since $\vec F_\infty$ is locally isoperimetric in $\bar \Pi$, 
    and $\bar \Pi$ is convex,
    we know that $T$ is locally minimizing with respect to variations, preserving its boundary $\partial T$, locally on the whole space $\RR^d$.
    So, $\partial T$ has $C^{1,\alpha}$ regularity and we can apply the boundary regularity 
    result in \cite{HarSim79} (see also \cite{All75})
    which states that $T$ is itself $C^{1,\alpha}$ regular in some neighbourhood of $\partial T$.
    In particular if $\partial T\neq 0$ we know that $T$ has a tangent hyperplane in $0$ and, being a cone, 
    it must be supported on that hyperplane.
    \end{itemize}
\end{proof}



\begin{proposition}\label{prostima}
Let $\Pi$ be a half-space of $\RR^d$ with $0\in \partial \Pi$.
Suppose that $\vec F=(F_1,F_2,F_3)$ is a $3$-partition of $\RR^d$, isoperimetric in $\bar \Pi$, 
with $\abs{F_1}=\abs{F_1\cap \Pi} < +\infty$ and $\abs{F_2}=\abs{F_3}=+\infty$.
Let $\vec F_\infty$ be as in Corollary \ref{corinfty}, for a sequence $r_k\to +\infty$.
Suppose one of the following:
\begin{enumerate}
\item[(i)] $\abs{F_2\setminus \Pi}=0$ or $\abs{F_3\setminus \Pi}=0$;
\item[(ii)] or $\vec F \setminus \bar \Pi = (\emptyset, H\setminus\bar\Pi, (\RR^d\setminus H)\setminus \bar \Pi)$,
where $\partial H$ is a hyperplane with $0\in \partial H$. 
\end{enumerate}
Then we have
\begin{gather}\label{ehia}
\Abs{\vec m ((\vec F\triangle (\vec F_\infty +\vec y))\cap (B_r\setminus B_{r-1})\cap \Pi)} \to 0
 \text{ as $r\to +\infty$, for some $\vec y\in\RR^d$},\\\label{ehio}
 \Delta_{\vec F}^{\bar \Pi} \ge \Delta_{\vec L},
\end{gather}
where $\vec L$ is a lens partition.
\end{proposition}

\begin{proof} 
The proof is similar in the two cases, we only present the case (ii) which is slightly more difficult.
Choose a suitable coordinate system in $\RR^d$ such that $\Pi=\{x_d<0\}$ and $\partial \Pi\cap\partial H=\{x_d=x_{d-1}=0\}$.\\
Given a set $G\subset \Pi$ we let $\sigma(G):=\{x\in \RR^d:\, (x_1,\ldots , x_{d-2},-x_{d-1},-x_d)\not\in G\}\subset \RR^d\setminus\bar\Pi$.
Using standard density estimates, as in \cite[Theorem~2.4]{NovPaoTor25}, one can prove that there exists $R>0$ such $\abs{F_1\setminus B_R}=0$.
Consider now the set
\[
E\defeq (F_2\cap \Pi) \cup \sigma(F_2\cap \Pi),
\]
and observe that $E$ is locally minimal both in $\Pi\setminus \bar B_R$ (since $\vec F$ is isoperimetric and $F_1$ is contained in $B_R$) 
and in $\RR^d\setminus(\bar\Pi\cup \bar B_R)$.
Moreover, on $\partial \Pi \setminus \bar B_R$ the set $\partial E$ is regular because 
$\partial \vec F$ is regular in a neighborhood of $\partial \Pi$. 
As a consequence $\vec v(\partial E\setminus \bar B_R,1)$ is a stationary varifold in $\RR^d\setminus \bar B_R$.
By Theorem \ref{teoallard} $\partial E$, hence also $\partial\vec F$, is asymptotically flat with the expansion in \eqref{expansion},
which shows in particular that $\partial \vec F_\infty\cap\Pi$ is a half-hyperplane and \eqref{ehia} holds. 
Consider the half-space $S$ such that $S \cap \Pi = \vec F_\infty^2 \cap \Pi$
and consider the Steiner symmetrical set $F'_1$ of $F_1$ with respect to the hyperplane $\partial S$. 
We consider the partition $\vec F' \defeq (F'_1,S\setminus F'_1,\RR^d\setminus (S\cup F'_1))$
and proceed as in \cite[Theorem 2.9]{BroNov24}:
following \cite[Lemma 5.1]{BroNov24}
one gets $\Delta_{\vec F'} \le \Delta_{\vec F}^{\bar \Pi}$,
then considering the capillarity problem in $S$ with prescribed volume $\frac 1 2 \abs{F_1}$, 
one obtains $\Delta_{\vec L} \le \Delta_{\vec F'}$.
This completes the proof.
\end{proof}

\begin{remark}
We expect that a partition $\vec F$ as in the proposition above is necessarily a lens partition, intersected with $\Pi$.
This would follow as in \cite[Theorem 2.9]{BroNov24} if we knew that $\partial\vec F_\infty$ is not contained in $\partial \Pi$.
On the other hand, there might exist a isoperimetric partition $\vec F$ in $\Pi$ with $\partial\vec F_\infty=\partial\Pi$, 
which is not a standard lens partition. 
\end{remark}

We are now ready to prove our main result.

\begin{theorem}[main result]\label{teomain}
There exists $\vec E=(E_1,E_2,E_3)$ an isoperimetric $3$-partition of $\RR^8$ with $\vec m(\vec E) = (1,+\infty,+\infty)$ 
which is not a lens partition (Definition~\ref{def:lens-partition}).
In particular its blow-down partition $\vec E_\infty=(\emptyset, E^1_\infty, E^2_\infty)$ is a singular minimal cone.
\end{theorem}

\begin{proof}
Let $R_n\to +\infty$ be any sequence of positive numbers with $\abs{B_{R_n}} > 1$. 
Let $\vec E_n=((E_n)_1,(E_n)_2,(E_n)_3)$, be an isoperimetric partition in the ball $\bar B_{R_n}\subset \RR^d$ 
with $\abs{(E_n)_1}=1$, 
such that $\vec E \triangle \vec S \subset \bar B_{R_n}$ where 
$\vec S=(\emptyset, S_2,S_3)$ is the Simons' partition 
(see Definition~\ref{def:simons-partition}).

By Theorem~\ref{teocon}, up to a non-relabeled subsequence, there exist $\vec x^n_i\in \RR^8$ and $\vec F^i=(F^i_1,F^i_2,F^i_3)$ 
such that
\begin{gather}
    \lim_{n\to+\infty} \abs{\vec x_n^i - \vec x_n^j} = +\infty \qquad \text{if $i\neq j$,}\\
    \vec E_n - \vec x_n^i \stackrel{L^1_\loc}\longrightarrow \vec F^i, \qquad \text{as $n\to+\infty$, for all $i\in \NN$}\\
    \label{eq:92554}
    \sum_{i\in \NN} \abs{F_1^i} = \abs{(E_n)_1} = 1.
\end{gather}
Without loss of generality we can assume $\abs{F_1^i}>0$
for all $i$.

Up to a subsequence, 
we assume that $B_{R_n}(-\vec x_n^i) \to \Pi^i$ in $L^1_\loc$ as $n\to +\infty$,
where either $\Pi^i=\RR^d$ or $\bar\Pi^i$ is a half-space of $\RR^d$.
The limit $\Pi^i$
cannot be empty because $\bar \Pi^i\supset F_1^i$ 
and we are assuming that $\abs{F_1^i}>0$.
By Theorem~\ref{th:teoclosure} for each $i\in \NN$ the partition $\vec F^i$ is isoperimetric in $\bar \Pi^i \subset \RR^d$.

If for some $i\in \NN$ we have $\Pi^i=\RR^d$, 
$\abs{F^i_2}=\abs{F^i_3}=+\infty$, 
and the partition $\vec F^i$ is not isometric to a lens partition,
then the proof is concluded by taking as $\vec E$ an appropriate rescaling of $\vec F^i$.
Indeed, by~\cite{BroNov24}, an isoperimetric $3$-partition of $\RR^d$ with only one finite region is a lens partition if its blow-down is a plane, 
we know that our non-lens partition $\vec E$ has a blow-down $\vec E_\infty$ which is a singular minimal cone.

Hence, we assume by contradiction that for all $i\in \NN$ either $\vec F^i$ is a lens partition, or $\Pi^i \neq \RR^d$,
or $\abs{F^i_j}<+\infty$ for $j=2$ or $j=3$.
We claim that in any case we have
\begin{equation}\label{eq:38443}
\Delta_{\vec F^i}^{\bar \Pi^i}
\ge \Delta_{\vec L}.
\end{equation}

If $\vec F^i$ is a lens partition then~\eqref{eq:38443} trivially holds.

If $\Pi^i\neq \RR^d$ it means that the sequence $\vec x_n^i$ stays close to $\partial B_{R_n}$ and in the limit (as $R_n\to+\infty$) we 
see an affine half-space. 
If the sequence also stays close to the cone $\partial \vec S$, passing to the limit the cone becomes a hyperplane 
$T^i$ perpendicular to $\partial \Pi^i$ and $\partial \vec F^i$ coincides with such hyperplane outside $\bar \Pi^i$.
Without loss of generality we can assume that $0\in T^i\cap \partial \Pi^i$.
We can apply Theorem~\ref{teoflat}, and conclude that the blow-down 
$\vec F_\infty^i$ of $\vec F^i$ is a half-space also 
inside $\Pi^i$, so $\partial \vec F_\infty^i$ is the union of 
two half-hyperplanes containing $T^i\cap \partial \Pi^i$.
Then we can apply Proposition~\ref{prostima} (ii) to obtain~\eqref{eq:38443}.
Another case is when $\bar \Pi^i$ is a half-space, $\abs{F_2^i}=\abs{F_3^i}=+\infty$, and $\abs{F_2^i\setminus \Pi}=0$ or $\abs{F_3^i\setminus \Pi}=0$; then by 
Proposition~\ref{prostima} (i) we obtain again~\eqref{eq:38443}.

The last case to consider is when either $\abs{F^i_2}<+\infty$ or $\abs{F^i_3}<+\infty$. If for instance $\abs{F^i_3}<+\infty$, it follows
that $\abs{F^i_3\setminus\bar \Pi^i}=0$, and $(F^i_1,F^i_2\cup F^i_3)$ is an isoperimetric cluster.
In this case $F^i_1$ is a ball $B_\rho$ and we have 
\[
\Delta_{\vec F^i} =\Delta_{\vec F^i}^{\bar \Pi^i}= \frac{\Per(B_\rho)}{\abs{B_\rho}^{\frac 7 8}} 
= \frac{8 \sqrt \pi}{\sqrt[8]{24}}
\approx 9.53 > \Delta_{\vec L}.
\]
Hence \eqref{eq:38443} is proven.

Now we can apply Lemma~\ref{lemmamain}
to obtain
\[
 \limsup_n \Delta_{\vec E_n} \ge \Delta_{\vec L}.
\]
But when $R_n$ is sufficiently large the 
barrel partition $\vec Q$ is a competitor 
to $\vec E_n$ and hence we must have $\Delta_{\vec E_n} \le \Delta_{\vec Q}$
which would give $\Delta_{\vec L}\le \Delta_{\vec Q}$ in contrast to the statement of Lemma~\ref{lm:defect-partition}.
\end{proof}

\begin{remark}
As observed in the Introduction, an interesting open question is whether $\vec E_\infty$ coincides with the Simons' cone.
We expect that this is the case, but a proof of this fact would require a more refined analysis.
\end{remark}
\appendix

\section{Concentration compactness}\label{secapp}

In this section we collect some results about concentration compactness for sequences of sets with finite perimeter in the Euclidean space.
A generalization of these results were proved in \cite{NoPaStTo22} in a more abstract framework.

\begin{lemma}[compactness in $L^1_\loc$]
\label{lm:L1-loc-compactness}
Let $E_n\subset \RR^d$ be a sequence of Caccioppoli sets such that 
\begin{equation}
\label{eq:fjsk}
 \sup_{\vec x\in \RR^d} \sup_{n\in \NN} \Per(E_n,B_1(\vec x)) < +\infty.
\end{equation}
Then there exists a Caccioppoli set $E\subset \RR^d$ such that, up to a subsequence 
$E_n\to E$ in $L^1_\loc$.
\end{lemma}

\begin{proof}
Assumption~\eqref{eq:fjsk} implies that for all $r>0$
\[
  \sup_n \Per(E_n, B_r) < +\infty.
\]
So, there is a subsequence such that $E_n\cap B_1$ converges in $L^1$, up to a subsequence 
we can assume that also $E_n\cap B_2$ converges, and so on. With a diagonal argument we 
can assume that $E_n \cap B_k$ converges in $L^1$ as $n\to+\infty$ for all $k$, which means 
that the sequence $E_n$ converge in $L^1_\loc$.
\end{proof}

\begin{lemma}[equisummability]
  \label{lm:equisummability}%
Suppose that $m_{k,j}\geq 0$, and
\begin{gather*}
\lim_{n\to+\infty} \sup_{k\in \NN} \sum_{j=n}^{+\infty} m_{k,j} = 0.
\end{gather*}
Then
\[
\sum_{j=1}^{+\infty} \lim_{k\to +\infty} m_{k,j} = \lim_{k\to+\infty} \sum_{j=0}^{+\infty} m_{k,j}.
\]
\end{lemma}
\begin{theorem}[concentration compactness]
\label{th:concentration}
Let $E_k$ be a sequence of measurable subsets of $\RR^d$ such that 
\begin{equation}
\label{eq:33948}
\sup_k \Per(E_k)=:P < +\infty,
\qquad
  \limsup_k \abs{E_k} =:m < +\infty.
\end{equation}
Then, there exist measurable sets $F^i$, $i\in \NN$ and points $\vec x_k^i \in \RR^d$ such that,
up to a subsequence in $k$, one has
\begin{gather}
    \label{eq:cc-infty}
    \lim_k \abs{\vec x_k^i - \vec x_k^j} = +\infty \qquad \text{for all $i\neq j$},\\
    \label{eq:cc-converge}
    E_k - \vec x_k^i \stackrel{L^1_\loc}\longrightarrow F^i \qquad \text{for all $i$ as $k\to+\infty$},\\
    \label{eq:cc-sum}
    \sum_{i\in \NN} \abs{F_i} = \lim_k \abs{E_k}.
\end{gather}
\end{theorem}
\begin{proof}
Since the statement is invariant under rescaling, without loss of generality we can 
assume that $m <\frac 12$ and hence that $|E_k|<\frac 12$ for every $k$.
For each $k$ let $\vec y_k^j$ be an enumeration of the lattice $\ZZ^d$ 
such that defining $Q_k^j = [0,1]^d + \vec y_k^j$ 
the sequence  
$j\mapsto \abs{E_k \cap Q_k^j}$ is decreasing.
If there are infinitely many cubes in which $E_k$ 
has positive measure we avoid to enumerate the cubes 
with $0$ measure.

Since $\sup_k P(E_k)<+\infty$, 
by Lemma~\ref{lm:L1-loc-compactness},
passing to a subsequence, for each $j$ we can suppose that 
$E_k - \vec y_k^j$ converges in $L^1_\loc$ and hence
there exists $G^j\subset [0,1]^d$ such that 
\[
  (E_k \cap Q_k^j) - y_k^j = (E_k - y_k^j)\cap [0,1]^d \stackrel{L^1}\longrightarrow G^j,
  \qquad \text{as $k\to \infty$}.
\]
Note that $\abs{G^j} = \lim_k \abs{E_k \cap Q_k^j}$.

Let $\gamma$ be the local isoperimetric constant, such that if $\abs{E\cap [0,1]^d} \le \frac 1 2$ 
then $\abs{E\cap[0,1]^d} \le \gamma \Per(E,[0,1]^d)^{1-\frac 1 d}$.
We have 
\begin{align*}
  \sum_{j=n}^{+\infty} \abs{E_k \cap Q_k^j}
  &= \sum_{j=n}^{+\infty} \abs{E_k \cap Q_k^j}^{\frac 1 d} 
  \cdot \abs{E_k \cap Q_k^j}^{\frac {d-1} d} 
  \le \gamma\sum_{j=n}^{+\infty} \abs{E_k \cap Q_k^j}^{\frac 1 d} 
  \Per(E_k, Q_k^j)\\
  &\le \gamma \abs{E_k \cap Q_k^n}^{\frac 1 d}\sum_{j = n}^{+\infty} \Per(E_k,Q_k^j)
  \le \gamma\enclose{\frac{\abs{E_k}} n}^{\frac 1 d} \sum_{j=n}^{+\infty}  \Per(E_k,Q_k^j)\\
  &\le \gamma\enclose{\frac{\abs{E_k}} n}^{\frac 1 d}  \Per(E_k)
  \le \frac{\gamma P}{2^{\frac 1 d}n^{\frac 1 d}}.
\end{align*}
Hence  
\[
\lim_{n\to+\infty} \sup_k \sum_{j=n}^{+\infty} \abs{E_k \cap Q_k^j} = 0.
\]
By Lemma~\ref{lm:equisummability} we have,
\[
    \sum_j \abs{G^j} 
    = \sum_j \lim_k \abs{E_k \cap Q_k^j}
    = \lim_k \sum_j \abs{E_k \cap Q_k^j}
    = \lim_k \abs{E_k}.
\]
Up to a subsequence we might suppose that, for 
all $j,j'\in \NN$, the limit $\lim_k \abs{\vec y_k^j-\vec y_k^{j'}}$
exists in $[0,+\infty]$.
So we can define an equivalence relation on $\NN$, by letting 
$j\sim j'$ whenever $\lim_k \abs{\vec y_k^j - \vec y_k^{j'}}<+\infty$
and let $I=\NN/{\sim}$ be the quotient set.
For each $i\in I$ choose a representant $j_i \in i$.

For every equivalence class $i$, and $j\in i$, the sequence 
$\vec y_k^j - \vec y_k^{j_i}$ is bounded in $\ZZ^d$, hence it is constant for $k$ large enough. 
Up to subsequences we may assume that for every equivalence class $i$, and $j\in i$, these sequences are constant: 
$\vec y_k^j - \vec y_k^{j_i} = \vec y^{j,i}\in \ZZ^d$. 
Let us denote $\vec y^{j,i}$ by $\vec y^j$ if $j\in i$.
Clearly for $j\in i$ all $\vec y^j$ are distinct
because $\vec y_k^j\in \ZZ^d$ are all distinct.
Let $\vec x_k^i = \vec y_k^{j_i}$, so that $\vec y^j=\vec y^j_k-\vec x^i_k$.

Let $F^i$ be the $L^1_\loc$ limit of $E_k - \vec x_k^i$. 
Since 
$(E_k - \vec y^j_k)\cap [0,1]^d \to G^j$, if $j\in i$
\begin{align*}
  \abs{G^j} 
    &= \lim_k \abs{E_k \cap Q_k^j}
    = \lim_k \abs{E_k \cap ([0,1]^d + \vec y_k^j)}\\
    &= \lim_k \abs{(E_k-\vec x_k^i)\cap([0,1]^d + (\vec y_k^j-\vec x_k^i))}\\
    &= \lim_k \abs{(E_k-\vec x_k^i)\cap([0,1]^d + \vec y^j)}
    = \abs{F^i\cap([0,1]^d + \vec y^j)}.
\end{align*}
So,
\begin{align*}
  \lim_k |E_k|
  & = \sum_j \abs{G^j} = \sum_i \sum_{j\in i} \abs{F^i\cap([0,1]^d + \vec y^j)}
  = \sum_i \abs{F^i\cap \bigsqcup_{j\in i}([0,1]^d + \vec y^j)}\\
  &\le \sum_i \abs{F^i}.
  \end{align*}
The other inequality follows from the semicontinuity of the measure 
with respect to the $L^1_\loc$ convergence.
\end{proof}

\bibliographystyle{amsplain}

\providecommand{\bysame}{\leavevmode\hbox to3em{\hrulefill}\thinspace}
\providecommand{\MR}{\relax\ifhmode\unskip\space\fi MR }
\providecommand{\MRhref}[2]{%
  \href{http://www.ams.org/mathscinet-getitem?mr=#1}{#2}
}
\providecommand{\href}[2]{#2}

\end{document}